\documentclass[a4paper,12pt]{article}
\usepackage{amsmath, amsthm, amsfonts, amssymb}
\usepackage{enumitem}
\usepackage{fullpage}

\newtheorem{theorem}{Theorem}[section]
\newtheorem{lemma}[theorem]{Lemma}

\newtheorem{corollary}[theorem]{Corollary}
\newtheorem{conjecture}[theorem]{Conjecture}
\newtheorem{question}[theorem]{Question}
\theoremstyle{definition}
\newtheorem{definition}[theorem]{Definition}

\newcommand{\eps}{\varepsilon}
\newcommand{\sm}{\setminus}
\newcommand{\1}{{\bf 1}}
\newcommand{\T}{\mathcal{T}}
\newcommand{\U}{\mathcal{U}}
\newcommand{\floor}[1]{\left\lfloor #1 \right\rfloor}
\newcommand{\ceiling}[1]{\left\lceil #1 \right\rceil}
\newcommand{\dir}[1]{\overrightarrow{#1}}

\DeclareMathOperator{\RT}{\mathbf{RT}}
\DeclareMathOperator{\ER}{ER}
\DeclareMathOperator{\BE}{BE}

\title{Triangle-tilings in graphs without large independent sets}
\author{J\'ozsef Balogh\thanks{Department of Mathematical Sciences,
University of Illinois at Urbana-Champaign, Urbana, Illinois 61801, USA. {\tt
jobal@math.uiuc.edu}. Research is partially supported by NSA Grant H98230-15-1-0002, NSF
Grant
DMS-1500121 and Arnold O. Beckman Research Award (UIUC Campus Research Board 15006).},
Andrew McDowell\thanks{School of Mathematics, University of Birmingham, Birmingham, B15 2TT, United Kingdom. {\tt
a.j.mcdowell@bham.ac.uk}.},
Theodore Molla\thanks{Department of Mathematical Sciences,
 University of Illinois at Urbana-Champaign, Urbana, Illinois 61801, USA. {\tt
molla@illinois.edu}. Research is partially supported by NSF Grant
DMS-1500121.}, 
Richard Mycroft\thanks{School of Mathematics, University of Birmingham, Birmingham, B15 2TT, United Kingdom. {\tt
r.mycroft@bham.ac.uk}. Research is partially supported by EPSRC grant EP/M011771/1. \newline The authors are grateful to the BRIDGE strategic alliance between the University of Birmingham
and the University of Illinois at Urbana-Champaign. This research was conducted as part of the
�Building Bridges in Mathematics� BRIDGE Seed Fund project.}}
\date{\today}

\begin{document}
\maketitle

\begin{abstract}
  We study the minimum degree necessary to
  guarantee the existence of perfect and almost-perfect triangle-tilings 
  in an $n$-vertex graph~$G$ with sublinear independence number.
  In this setting, we show that if $\delta(G) \ge n/3 + o(n)$ then $G$ 
  has a triangle-tiling covering all but at most four vertices. Also, 
  for every $r \ge 5$, we asymptotically determine the minimum degree 
  threshold for a perfect triangle-tiling under the additional assumptions 
  that~$G$ is $K_r$-free and $n$ is divisible by $3$.
\end{abstract}

\section{Introduction}

A \textit{triangle-tiling} in a graph $G$ is a collection $\T$ of vertex-disjoint triangles in $G$. We say that $\T$ is \textit{perfect} if $|\T| = n/3$, where $n$ is the order of $G$. A trivial necessary condition for the existence of a perfect triangle-tiling is that 3 divides $n$.
We let $V(\T) := \bigcup_{T \in \T} V(T)$ and say $\T$ \textit{covers} $U \subseteq V(G)$ (respectively $v \in V(G)$) 
when $U \subseteq V(\T)$ (respectively $v \in V(\T)$),
so a perfect triangle-tiling covers every vertex of the host graph. 
Given disjoint sets $A$ and $B$ which partition~$V(G)$, we say that a triangle $T$ in $G$ is an \emph{$A$-triangle} if $T$ contains two vertices of~$A$ and one vertex of $B$, and likewise that $T$ is a \emph{$B$-triangle} if $T$ contains two vertices of $B$ and one vertex of $A$.
Observe that if $|A| = 1 \pmod 3$ and $|B| = 2 \pmod 3$, there are no $B$-triangles in $G$ and also there is no pair of vertex-disjoint $A$-triangles in $G$, then $G$ does not have a perfect triangle-tiling.
In that case, we call the ordered pair $(A, B)$ a \emph{divisibility barrier} in $G$ (note that order is important here).
Similarly, if $A \subseteq V(G)$ has size $|A| \geq 2n/3 + r$ for some $r > 0$, but $G[A]$ has no triangles, then every triangle-tiling in $G$ contains at most $n - |A| \leq n/3 - r$ triangles, and so leaves at least $3r$ vertices uncovered. We call such a set $A$ a \emph{space barrier}.

The classical Corr\'adi-Hajnal theorem~\cite{CH} states that if~$G$ has minimum degree $\delta(G) \geq 2n/3$, and~$n$ is divisible by~$3$, then~$G$ contains a perfect triangle-tiling. The minimum degree condition of this result is easily seen to be best-possible by considering, for an arbitrary $m \in \mathbb{N}$, the complete tripartite graph $G_1(m)$ with vertex classes of size $m-1, m$ and $m+1$. Indeed, $G_1(m)$ then has $n := 3m$ vertices and $\delta(G_1(m)) \geq 2m-1 = 2n/3-1$, but $G_1(m)$ has no perfect triangle-tiling, as the union of the two largest vertex classes is a space barrier. Observe, however, that $G_1(m)$ contains large independent sets.
By proving the following theorem, Balogh, Molla and Sharifzadeh~\cite{BMS} recently showed that the minimum degree condition can be significantly weakened if we additionally assume that~$G$ has no large independent set. Throughout this paper we write~$\alpha(G)$ to denote the independence number of~$G$.

\begin{theorem}[{\cite[Theorem~1.2]{BMS}}] \label{BMSthm}
For every $\omega > 0$ there exist $n_0, \gamma > 0$ such that the following holds for every integer $n \geq n_0$ which is divisible by 3. If $G$ is a graph on $n$ vertices with $\delta(G) \geq n/2 + \omega n$ and $\alpha(G) \leq \gamma n$, then $G$ contains a perfect triangle-tiling.
\end{theorem}

For an arbitrary $m \in \mathbb{N}$, the graph $G_2(m)$ consisting of two copies of $K_{3m+2}$ intersecting in a single vertex has $n := 6m+3$ vertices, minimum degree $\delta(G_2(m)) = 3m + 1 = \floor{n/2}$ and independence number two. Moreover, $G_2(m)$ has a divisibility barrier $(A, B)$, where $B$ is the vertex set of one of the copies of $K_{3m+2}$ and $A = V(G_2(m))\sm B$. This example demonstrates that the minimum degree condition of Theorem~\ref{BMSthm} is best-possible up to the $\omega n$ additive error term. 
Noga Alon suggested that if one only wants a triangle-tiling that covers all but a constant number
of vertices, then perhaps the condition $\delta(G) \ge (1/3 + o(1))n$ is sufficient.
In this paper, we show that this is indeed the case, by proving that if $\delta(G) \ge (1/3 + o(1))n$ and $\alpha(G) = o(n)$, 
then $G$ has a triangle-tiling covering all but at most four vertices. Furthermore, under the additional assumptions that $G$ has no divisibility barrier and 3 divides $n$, we show that $G$ contains a perfect triangle-tiling. 

\begin{theorem} \label{main}
For every $\omega > 0$ there exist $n_0, \gamma > 0$ such that if $G$ is a graph on $n \geq n_0$ vertices with $\delta(G) \geq n/3 + \omega n$ and $\alpha(G) \leq \gamma n$, then 
\begin{enumerate}[label=(\alph*), noitemsep]
\item $G$ contains a triangle-tiling covering all but at most four vertices of $G$, and
\item if $3$ divides $n$ and $G$ contains no divisibility barrier, then $G$ contains a perfect triangle-tiling.
\end{enumerate}
\end{theorem}

Observe that for an arbitrary $m \in \mathbb{N}$, the graph $G_3(m)$ consisting of two disjoint copies of $K_{3m+2}$ has $n := 6m+4$ vertices, minimum degree $\delta(G_3(m)) = 3m + 1 = n/2-1$ and independence number two, but every triangle-tiling in $G_3(m)$ covers at most $n-4$ vertices. This demonstrates that the conditions of Theorem~\ref{main} do not guarantee a triangle-tiling which leaves fewer than four vertices uncovered. Furthermore, in Section~\ref{sec:examples} we give a construction showing that the $\omega n$ error term in the minimum degree condition of Theorem~\ref{main} cannot be removed completely.

The relationship between the results in this paper and the Corr\'adi-Hajnal theorem
is clearly analogous to the relationship between Ramsey-Tur\'an theory and Tur\'an's theorem,
as Ramsey-Tur\'an theory is concerned with the maximum possible number of edges in 
an $H$-free graph on $n$ vertices with some upper bound on $\alpha(G)$.
More precisely, in classical Ramsey-Tur\'an theory the principle object of study is the function 
$\RT(n,H,m)$, which is defined 
to be the maximum number of edges in an $H$-free, 
$n$-vertex graph with independence number at most $m$, 
whenever such a graph exists for $n$, $H$ and $m$.
The asymptotic value of $\RT(n,K_r,o(n))$ was established for odd $r$ by 
Erd\H{o}s and S\'os~\cite{ES} and for even $r$ by
Erd\H{o}s, Hajnal, S\'os and Szemer\'edi~\cite{EHSS}, giving the following theorem.

\begin{theorem}[{\cite[Theorem 1]{ES} and~\cite[Theorem 1]{EHSS}}]\label{rt-thm}
  For every $r \ge 3$, we define
  \begin{equation*}
    f_{\text{RT}}(r) := 
    \begin{cases}
      \frac{r - 3}{r - 1}    & \text{if $r$ is odd,}\\
      \frac{3r - 10}{3r - 4} & \text{if $r$ is even}.
    \end{cases}
  \end{equation*}
  \begin{enumerate}[label=(\alph*), noitemsep]
    \item
      For every $\omega > 0$, there exists $\gamma, n_0 > 0$ such that
      if $G$ is a graph on $n \geq n_0$ vertices with $\alpha(G) \le \gamma n$ and with 
      at least $(f_{\text{RT}}(r)+\omega) \binom{n}{2}$ edges, then $G$ contains a copy of $K_r$. 
    \item
      For every $\omega > 0$ and $\gamma > 0$, there exists $n_0 > 0$ such
      that for every $n \ge n_0$, there exists a $K_r$-free graph $G := G_{\text{RT}}(n, r, \omega, \gamma)$
      on $n$ vertices such that $\delta(G) \ge (f_{\text{RT}}(r) - \omega) n$ and $\alpha(G) \le \gamma n$.
  \end{enumerate}
\end{theorem}

Observe that for any $r \ge 3$, $\omega, \gamma > 0$ and each sufficiently large $n$ divisible by $6$, the graph $G_4(n)$
on $n$ vertices consisting of the disjoint union of $G_{\text{RT}}(\frac{n}{2} - 1, r, \omega, \gamma)$ and
$G_{\text{RT}}(\frac{n}{2} + 1, r, \omega, \gamma)$ is $K_r$-free,
has minimum degree $\delta(G_4(n)) \geq \left(\frac{f_{\text{RT}}(r)}{2} - \omega\right)n$
and independence number at most $\gamma n$. However, as $G_4(n)$ 
contains a divisibility barrier, it has no perfect triangle-tiling. 
Although the construction of $G_{\text{RT}}(n, r, \omega, \gamma)$ was given in
\cite{ES} (when $r$ is odd) and~\cite{EHSS} (when $r$ is even),
for completeness, we describe $G_{\text{RT}}(n, r, \omega, \gamma)$ at the end of Section~\ref{sec:examples}.

By combining Theorems~\ref{main} and~\ref{rt-thm} we determine, for every $r \geq 5$, the asymptotic minimum degree threshold for a perfect triangle-tiling in a $K_r$-free graph with sublinear independence number. 
Indeed, Corollary~\ref{krfree} does this for $r \geq 7$, and the thresholds for $r=5$ and $r=6$ follow, as discussed after the proof.

\begin{corollary} \label{krfree}
For every $r \ge 7$ and $\omega > 0$ there exist $n_0, \gamma > 0$ such that the following holds for every integer $n \geq n_0$ which is divisible by 3. 
If $G$ is a $K_r$-free graph on $n$ vertices with 
$\delta(G) \geq \frac{f_{\text{RT}}(r)}{2}n + \omega n$ and $\alpha(G) \leq \gamma n$, then $G$ contains a perfect triangle-tiling.
\end{corollary}

\begin{proof}
Given $\omega > 0$, choose $\gamma$ small enough and $n_0$ large enough to apply Theorem~\ref{main} with the same constants there as here and so that we may apply Theorem~\ref{rt-thm}(a) with $3\gamma$ and $n_0/3$ in place of $\gamma$ and $n_0$ respectively. We also insist that $\gamma n_0 + 2 \leq \omega n_0/2$. Since $\frac{f_{\text{RT}}(r)}{2} \geq \frac{1}{3}$ for $r \geq 7$, by Theorem~\ref{main}(b) it suffices to prove that no $K_r$-free graph on $n \geq n_0$ vertices with $\delta(G) \geq \frac{f_{\text{RT}}(r)}{2}n + \omega n$ and $\alpha(G) \leq \gamma n$ contains a divisibility barrier. So let $G$ be such a graph, and suppose for a contradiction that $(X, Y)$ is a divisibility barrier in $G$. Let $A$ be the smaller of $X$ and $Y$, and let $B$ be the larger, so $|A| \leq n/2$. By definition of a divisibility barrier, if $A = Y$ then there is no pair of vertex-disjoint $B$-triangles in $G$, whilst if $A = X$ then there are no $B$-triangles in $G$ at all. It follows that at most one vertex $a \in A$ has more than $\gamma n + 2$ neighbours in $B$, as given two such vertices $a, a' \in A$ we could use the fact that $\alpha (G) \leq \gamma n$ to choose an edge $bc$ in $N(a) \cap B$ and then an edge $b'c'$ in $(N(a') \cap B) \sm e$ to obtain a pair of vertex-disjoint $B$-triangles $abc$ and $a'b'c'$ in $G$. So at least $|A| - 1$ vertices of $A$ have at least $\delta(G) - \gamma n - 2 \geq \frac{f_{\text{RT}}(r)}{2}n + \frac{\omega}{2}n$ neighbours in $A$. So in particular $|A| \geq \frac{f_{\text{RT}}(r)}{2}n \geq \frac{n}{3}$. Moreover we have
\begin{equation*}
  e(G[A]) \geq \frac{1}{2}(|A| - 1)\left(\frac{f_{\text{RT}}(r)}{2} + \frac{\omega}{2}\right)n
  = \frac{n}{2|A|}\left(f_{\text{RT}}(r) + \omega\right)\binom{|A|}{2} \geq
   \left(f_{\text{RT}}(r) + \omega \right) \binom{|A|}{2},
\end{equation*}
so $G[A]$ contains a copy of $K_r$ by Theorem~\ref{rt-thm}(a). This contradicts our assumption that $G$ was $K_r$-free and so completes the proof.
\end{proof}

Note that $\frac{f_{\text{RT}}(5)}{2} < \frac{f_{\text{RT}}(7)}{2} = \frac{1}{3}$. In Section~\ref{sec:examples} we give a construction of a $K_5$-free graph on $n$ vertices with minimum degree at least $n/3$ and sublinear independence number which contains a space barrier.
This demonstrates that the minimum degree condition in Corollary~\ref{krfree} is best-possible up to the $\omega n$ error term for $r=7$, and cannot be lowered 
by requiring $G$ to be $K_5$-free as opposed to $K_7$-free. Furthermore, the graph $G_4(n)$ presented after Theorem~\ref{rt-thm} shows that the minimum degree condition in Corollary~\ref{krfree} is best-possible up to the $\omega n$ error term for $r \geq 8$ also.

In a $K_4$-free graph, we can only construct space barriers when $\delta(G) < n/6$,
so it may be true that, in a $K_4$-free graph, the conditions $\delta(G) \geq (1/6 + o(1))n$ and $\alpha(G) = o(n)$ are sufficient
to guarantee a perfect triangle-tiling when $n$ is divisible by $3$;
we discuss this further in Section~\ref{sec:examples}. 
Also in Section~\ref{sec:examples}, we consider the problem of determining the minimum degree condition 
which guarantees a perfect $K_k$-tiling in a graph with sublinear independence number when $k \ge 4$.

\section{Notation and preliminary results}

In this section we introduce various results which we will use in the proof of Theorem~\ref{main}, beginning with helpful notation. We write $x = y \pm z$ to mean $y-z \leq x \leq y+z$, and $[n]$ to denote the set of integers from $1$ to $n$. We omit floors and ceilings throughout this paper wherever they do not affect the argument. We write $x \ll y$ to mean that for every $y > 0$ there exists $x_0 > 0$ such that the subsequent statements hold for $x$ and $y$ whenever $0 < x \leq x_0$. Similar statements with more variables are defined similarly.

\subsection{Regularity}

In a graph $G$, for each pair of disjoint non-empty sets $A, B \subseteq V(G)$ we write $G[A, B]$ for the bipartite subgraph of $G$ with vertex classes $A$ and $B$ and whose edges are all edges of $G$ with one endvertex in $A$ and the other in $B$, and denote the \emph{density} of $G[A, B]$ by $d_G(A, B) := \tfrac{e(G[A, B])}{|A||B|}$. We say that $G[A, B]$ is $(d,\eps)$-\emph{regular} if $d_G(X, Y) = d \pm \eps$ for every $X \subseteq A$ and $Y \subseteq B$ with $|X| \geq \eps |A|$ and $|Y| \geq \eps |B|$, and we write that $G[A, B]$ is $(\geq\!\!d, \eps)$-\emph{regular} to mean that $G[A, B]$ is $(d', \eps)$-regular for some $d' \geq d$. Also, we say that $G[A, B]$ is $(d, \eps)$-\emph{super-regular} if $G[A, B]$ is $(\geq\!\!d, \eps)$-regular, every vertex of $A$ has at least $(d-\eps)|B|$ neighbours in $B$ and every vertex of $B$ has at least $(d-\eps)|A|$ neighbours in $A$. The following well-known results are elementary consequences of the definitions.

\begin{lemma}[Slicing Lemma]\label{slicing}
For every $d, \eps, \beta > 0$, if $G[A, B]$ is $(d, \eps)$-regular, and $X \subseteq A$ and $Y \subseteq B$ have sizes $|X| \geq \beta |A|$ and $|Y| \geq \beta |B|$, then $G[X, Y]$ is $(d, \eps/\beta)$-regular.
\end{lemma}

\begin{lemma}\label{makesuperreg}
For every $d, \eps > 0$ with $\eps < \frac{1}{2}$, if $G[A, B]$ is $(\geq\!\!d, \eps)$-regular, then there are sets $X \subseteq A$ and $Y \subseteq B$ with sizes $|X| \geq (1-\eps) |A|$,  and $|Y| \geq (1-\eps) |B|$ such that $G[X, Y]$ is $(d, 2\eps)$-super-regular.
\end{lemma}

We make use of Chernoff bounds on the concentration of binomial and hypergeometric distributions in the following form.

\begin{theorem} [{\cite[Corollary 2.3 and Theorem 2.10]{JLR}}]\label{chernoff}
Suppose $X$ has binomial or hypergeometric distribution and $0<a<3/2$. Then
$\mathbb{P}(|X - \mathbb{E}X| \ge a\mathbb{E}X) \le 2
e^{-\frac{a^2}{3}\mathbb{E}X}$.
\end{theorem}

The following lemma is similar to lemmas of Csaba and Mydlarz~\cite[Lemma 14]{CM}
and Martin and Skokan~\cite[Lemma 10]{MS}. It states that if we randomly select a collection of disjoint subsets from each of the vertex classes of a super-regular pair, every pair of sets from different classes is super-regular with high probability. 

\begin{lemma}[Random Slicing Lemma]\label{random-slicing}
  Suppose that $1/n \ll \beta, \eps \ll d$.
  Let $G[A,B]$ be $(d, \eps)$-super-regular (respectively $(d, \eps)$-regular)
  where $|A|,|B| \le n$ and let  
  $x_1, \dotsc, x_s$ and $y_1, \dotsc, y_t$ be positive integers each of size at least $\beta n$ such that
  $\sum_{i \in [s]} {x_i} \le |A|$ and $\sum_{j \in [t]} {y_j} \le |B|$.
  If $\{X_1, \dotsc, X_s\}$ is a collection of disjoint subsets of $A$ and 
  $\{Y_1, \dotsc, Y_t\}$ is a collection of disjoint subsets of $B$ 
  such that $|X_i| = x_i$ and $|Y_j| = y_j$ for all $i \in [s]$ and $j \in [t]$ 
  selected uniformly at random from all such collections, 
  then, with probability at least $1 - e^{-\Omega(n)}$, 
  $G[X_i, Y_j]$ is $(d, \eps')$-super-regular (respectively $(d, \eps')$-regular) for all $i \in [s]$ and $j \in [t]$,
  where $\eps' := (33 \eps)^{1/5}$.
\end{lemma}

For completeness we present a proof of Lemma~\ref{random-slicing} in the Appendix.
To make use of regularity properties, we apply the degree form of Szemer\'edi's Regularity Lemma (see~\cite[Theorem~1.10]{KS96}).

\begin{theorem}[Degree form of Szemer\'edi's Regularity Lemma]\label{reglem}
  For every $\eps > 0$, real number $d \in [0, 1]$ and integers $t$ and $q$ there exists integers $n_0$ and $T$ such that the following statement holds. Let $G$ be a graph on $n \geq n_0$ vertices, and let $U_1, \dots, U_q$ be a partition of $V(G)$ into $q$ parts. Then there is a partition of $V(G)$ into an exceptional set $V_0$ and $k$ clusters $V_1, \dotsc, V_k$, and a spanning subgraph $G' \subseteq G$ such that
  \begin{enumerate}[label=(\alph*), noitemsep]
    \item $t \leq k \leq T$,
    \item $|V_1| = |V_2| = \ldots = |V_k|$ and $|V_0| \leq \eps n$, 
    \item for every $i \in [k]$ there exists $j \in [q]$ such that $V_i \subseteq U_j$, 
    \item $d_{G'}(v) \ge d_G(v) - (\eps + d)n$ for all $v \in V(G)$,
    \item $e(G'[V_i]) = 0$ for all $i \in [k]$, and
    \item for each distinct $i, j \in [k]$ either $G'[V_i, V_j]$ is $(\geq\!\!d, \eps)$-regular or $G'[V_i, V_j]$ is empty.
  \end{enumerate}
\end{theorem}

Theorem~\ref{reglem} as stated above is stronger than the form given in~\cite{KS96} in that it allows us to specify an initial partition of $V(G)$ and to insist that the clusters $V_1, V_2, \dots, V_k$ are each a subset of some part of this partition (property~(c) above). However, this statement follows from the same proof, which proceeds iteratively by alternately refining a partition of $V(G)$ and deleting some vertices of $V(G)$ (which are then placed in the exceptional set $V_0$). So to prove Theorem~\ref{reglem} we take our specified partition as the initial partition of this process.

\subsection{Robustly-matchable sets}

The following application of the regularity lemma is critical to the entire proof. 
Given a graph $G$, a small $A \subseteq V(G)$ and a small matching $B \subseteq E(G)$, we form an auxiliary bipartite graph 
$F$ with vertex set $A\cup B$ in which there is an edge between $a\in A$ and $bc\in B$ if and only if $abc$ is a triangle in $G$. So matchings in $F$ correspond to triangle-tilings in $G$.
In this setting, Lemma~\ref{absorber} allows us to choose subsets $X \subseteq A$ and $Y \subseteq B$ such that if we can find a triangle-tiling in $G$
that covers every vertex of $G$ except for the vertices incident to edges in $Y$ and
exactly $|Y|$ of the vertices in $X$, then we obtain a perfect triangle-tiling in $G$.

\begin{lemma} \label{absorber}
Suppose that $1/n \ll \phi \ll \eps \ll d$. Let $F$ be a bipartite graph with vertex classes $A$ and $B$ such that $n/10 \leq |A|, |B| \leq n$ and $d_F(A, B) \geq d$. Then there exist subsets $X \subseteq A$ and $Y \subseteq B$ of sizes $|X| = \phi n$ and $|Y| = (1-\eps) \phi n$ such that $F[X', Y]$ contains a perfect matching for every subset $X' \subseteq X$ with $|X'| = |Y|$.
\end{lemma}

\begin{proof}
Let $n_0$ and $T$ be the integers returned by Theorem~\ref{reglem} given inputs $\eps$, $d' := d/200$ and $t=q=2$. We may assume that $\phi \leq 1/4T$.
We use Theorem~\ref{reglem} with initial partition $U_1=A$ and $U_2 = B$ to obtain a spanning subgraph $F' \subseteq F$ and a partition of $V(F)$ into sets $V_0, V_1, \dots, V_k$ which satisfy properties~(a)--(f) of Theorem~\ref{reglem}. In particular, by Theorem~\ref{reglem}(d) at most $(\eps+d/200)n^2$ edges of $F$ are not edges of $F'$. Also, by Theorem~\ref{reglem}(e) there are no edges in $F'[V_i]$ for any $i \in [k]$, and since $|V_0| \leq \eps n$ by Theorem~\ref{reglem}(b), at most $\eps n^2$ edges of $F$ contain a vertex of $V_0$. Since
$$e(F) = d_F(A, B)|A||B| \geq d\left(\frac{n}{10}\right)^2 > \left(\eps+\frac{d}{200}\right)n^2 + \eps n^2,$$
there must exist distinct $i, j \in [k]$ such that $F'[V_i, V_j]$ is non-empty, and since $F$ is bipartite, by Theorem~\ref{reglem}(c) we may assume without loss of generality that $V_i \subseteq A$ and $V_j \subseteq B$.
Observe that $F'[V_i, V_j]$ is $(\geq\!\!d', \eps)$-regular by Theorem~\ref{reglem}(f). Write $m$ for the common size of $V_i$ and $V_j$, so 
$m = |V(F) \sm V_0|/k \geq n/2T \geq 2\phi n$ by Theorem~\ref{reglem}(a)~and~(b). 
By Lemma~\ref{makesuperreg} we may delete at most $\eps m$ vertices from each of $V_i'$ and $V_j'$ to obtain subsets $V'_i \subseteq V_i$ and $V_j' \subseteq V_j$ such that $F[V_i', V_j']$ is $(d', 2\eps)$-super-regular. Having done so, choose $X \subseteq V_i'$ and $Y \subseteq V_j'$ uniformly at random with sizes $\phi n$ and $(1-\eps)\phi n$ respectively (this is possible since $|V_i'|, |V_j'| \geq (1-\eps) m \geq \phi n$). Then Lemma~\ref{random-slicing} tells us that $F'[X, Y]$ is $(d', \eps')$-super-regular with high probability, where $\eps' := (66\eps)^{1/5}$, so we may fix sets $X$ and $Y$ with this property. It then follows that every vertex of $X$ has at least $(d'-\eps') |Y| \geq \eps' |X|$ neighbours in $Y$, whilst every set of at least $\eps' |X|$ vertices of $X$ has at least $(1-\eps') |Y| \geq (1-2\eps')|X|$ neighbours in $Y$ (where we say that a vertex $y$ is a neighbour of a set $X'$ if $y$ is a neighbour of some element of $X'$). Finally, since every vertex of $Y$ has at least $(d' - \eps') |X| > 2 \eps' |X|$ neighbours in $X$, every set of at least $(1-2\eps')|X|$ vertices of $X$ has every vertex of $Y$ as a neighbour. So Hall's criterion is satisfied for every $X' \subseteq X$ of size $|X'| \leq |Y|$, so for every $X' \subseteq X$ with $|X'| = |Y|$ there is a perfect matching in $F'[X', Y]$.
\end{proof}

\subsection{Spanning bounded degree trees}

Our proof requires us to find a spanning tree of bounded maximum degree in the reduced graph $R$ of $G$. For this, we use the following theorem of Win~\cite{Win}.

\begin{theorem}\label{spantree}
If $k \ge 2$ and $R$ is a connected graph such that 
\begin{equation*}
  \sum_{v \in S} d(v) \ge |R| - 1 
  \text{ for every independent set $S$ of size $k$},
\end{equation*}
then $R$ contains a spanning tree $T$ such that $\Delta(T) \le k$.
In particular, if $R$ is a connected graph with $\delta(R) \ge (|R|-1)/k$, then $R$ contains
a spanning tree $T$ with maximum degree at most $k$.
\end{theorem}

\subsection{Fractional weighted matchings via linear programming}

In our proof of Theorem~\ref{main}, we will consider regular pairs of clusters of vertices of $G$ and use the regularity of each pair to find a triangle-tiling covering a given proportion of vertices from each cluster. 
We want to choose these proportions so that collectively these triangle-tilings cover (almost) all of the vertices of $G$. To do this we look for a generalized form of weighted matching in the reduced graph; the proportion of vertices to be covered by a triangle-tiling within a pair of clusters then corresponds to the weight in this matching of the corresponding edge of the reduced graph. 

A \emph{fractional matching} $w$ in a graph $G$ assigns a weight $w_e \geq 0$ to each edge $e \in E(G)$ such that for every vertex $u \in V(G)$ we have $\sum_{e \ni u} w_e \leq 1$. In other words, if we consider each edge $uv$ to place weight $w_{uv}$ at each of $u$ and $v$, then the the combined weight placed at each vertex is at most one. This is a relaxation of an integer matching $M$, in which we insist that for each $e \in E(G)$ we have $w_e = 1$ (meaning that $e \in M$) or $w_e = 0$ (meaning that $e \notin M$). Here we work with a more general notion of an \emph{$(\eta, \xi)$-weighted fractional matching}, in which we consider each edge to place different weights at each end, subject to the restriction that the ratio of these weights is at most $\eta : \xi$. It is most natural to express these matchings in terms of directed graphs, as we can then consider a directed edge $\dir{uv}$ of weight $w_{\dir{uv}}$ to place weight $\eta w_{\dir{uv}}$ on its tail $u$ and weight $\xi w_{\dir{uv}}$ on its head $v$; as before, we insist that the combined weight placed at each vertex is at most one. 

\begin{definition}
Let $\Gamma$ be a directed graph on $n$ vertices and let $\eta$ and $\xi$ be positive real numbers. An \emph{$(\eta, \xi)$-weighted fractional matching} $w$ in $\Gamma$ is an assignment of a weight $w_{\dir{uv}} \geq 0$ to each edge $\dir{uv}$ of $\Gamma$ such that for every vertex $u \in V(\Gamma)$ we have 
\begin{equation} \label{eq:fracmatch}
\sum_{v \in N^+_{\Gamma}(u)} \eta w_{\dir{uv}} + \sum_{v \in N^-_{\Gamma}(u)} \xi w_{\dir{vu}} \leq 1.
\end{equation}
The \emph{total weight} of $w$ is defined to be $W := \sum_{\dir{uv} \in E(\Gamma)} (\eta + \xi)w_{\dir{uv}}$. By~\eqref{eq:fracmatch} we have $W \leq n$; we say that $w$ is \emph{perfect} if $W = n$. Note that in this case we have equality in~\eqref{eq:fracmatch} for every vertex.
\end{definition}

Given an undirected graph $G$, we consider $(\eta, \xi)$-weighted fractional matchings in the directed graph $\Gamma$ formed by replacing every edge $uv$ of $G$ with both a directed edge $\dir{uv}$ from $u$ to $v$ and a directed edge $\dir{vu}$ from $v$ to $u$.
In particular, a $(\frac{1}{2}, \frac{1}{2})$-weighted fractional matching $w$ in $\Gamma$ then corresponds to a fractional matching $w'$ in $G$ (in the standard notion of fractional matching as defined above). Indeed, given $w$, for each edge $e = uv \in E(G)$ we may take $w'_e = w_{\dir{uv}} + w_{\dir{vu}}$. 
In our proof we will instead consider $(\eta, \xi)$-weighted fractional matchings in $\Gamma$ where $\xi$ is close to twice as large as $\eta$. 
The advantage of this is shown by Lemma~\ref{wfracm}, which states that 
the minimum degree condition on $G$ needed to guarantee the existence of a perfect $(\eta, \xi)$-weighted fractional matching in $\Gamma$ is then approximately $n/3$, well below the $n/2$ threshold needed to guarantee the existence of a perfect fractional matching in $G$.

Let $\Gamma$ be a directed graph on $n$ vertices $v_1, \dots, v_n$, and fix $\eta, \xi > 0$. Then we define the \emph{$(\eta, \xi)$-weighted characteristic vector} of an edge $\dir{uv} \in E(\Gamma)$ to be the vector $\chi_{\eta, \xi}(\dir{v_iv_j}) \in \mathbb{R}^n$ whose $i$th coordinate is equal to $\eta$, whose $j$th coordinate is equal to $\xi$, and in which all other coordinates are equal to zero. So an assignment $w$ of non-negative weights to edges of $\Gamma$ is an $(\eta, \xi)$-weighted fractional matching in $\Gamma$ if and only if 
\begin{equation} \label{eq:fracmatch2}
  \sum_{\dir{v_iv_j} \in E(\Gamma)} w_{\dir{v_iv_j}} \chi_{\eta, \xi}(\dir{v_iv_j}) \leq \1,
\end{equation}
where $\1$ is the vector in $\mathbb R^n$ with each coordinate equal to $1$ and the inequality is treated pointwise. As before, $w$ is perfect if and only if we have equality for each coordinate.

To prove the existence of a $(\eta,\xi)$-weighted fractional matching in a directed graph of high minimum indegree, we use the following version of Farkas' Lemma,
for which we need the following definition;
a vertex $\mathbf v \in \mathbb{R}^n$ is a \emph{weighted sum} 
of vectors in $\mathcal{X} = \{\mathbf{x}_1, \dotsc, \mathbf{x}_m\} \subseteq \mathbb{R}^n$ if 
$$
\mathbf v \in \left\{ \sum_{i = 1}^{m} \lambda_i x_i : \lambda_i \ge 0 \text{ for every $i \in [m]$ } \right\},
$$
otherwise $\mathbf v$ is not a weighted sum of the vectors in $\mathcal X$.

\begin{lemma}[Farkas' Lemma]\label{farkas}
  For every $\mathbf v \in \mathbb{R}^n$ and every finite $\mathcal X \subseteq \mathbb{R}^n$,
  if $\mathbf v$ is not a weighted sum of the vectors in $\mathcal X$, then
  there exists $\mathbf y \in \mathbb{R}^n$ such that 
  $\mathbf{y} \cdot \mathbf x \ge 0$ for every $\mathbf x \in \mathcal X$
  and $\mathbf{y} \cdot \mathbf{v} < 0$.
\end{lemma}

We now give the main result of this section.

\begin{lemma}\label{wfracm}
  For every $\eta > 0$, every directed graph $\Gamma$ on $n$ vertices with $\delta^-(\Gamma) \geq \eta n$ 
  admits a perfect fractional $(\eta, 1 - \eta)$-matching. 
  Furthermore, if $\eta = p/q$ for positive integers $p$ and $q$, then we can assume that
  the weights of the matching are rational numbers with common denominator $D$ bounded above by some function of $p$, $q$ and $n$.
\end{lemma}

\begin{proof}
Let $v_1, \dotsc, v_n$ be an arbitrary ordering of the vertices of $\Gamma$. Then by~\eqref{eq:fracmatch2}, a perfect $(\eta,1 - \eta)$-weighted fractional matching in $\Gamma$ corresponds to a weighted sum of the vectors in
  \begin{equation*}
    \mathcal X := \{ \chi_{\eta,1-\eta}(\dir{v_iv_j}) : \dir{v_iv_j} \in E(\Gamma)\}
  \end{equation*}
  that equals $\mathbf 1$.

  If we assume that $\Gamma$ does not have a perfect $(\eta,1-\eta)$-weighted fractional matching, then, by Farkas' lemma (Lemma~\ref{farkas}), as $\mathbf 1$ is not a weighted sum of the vectors in $\mathcal X$, there exists a vector $\mathbf{y} \in \mathbb{R}^n$ such that $\mathbf{y} \cdot \mathbf{1} < 0$ but $ \mathbf{y} \cdot \chi_{\eta,1-\eta}(\dir{v_iv_j}) \geq 0$ for every $\dir{v_iv_j} \in E(\Gamma)$. By reordering the vertices if necessary, we may assume that $y_1 \geq \ldots \geq y_n$.

  Let $i$ be maximal such that $\dir{v_iv_n} \in E(\Gamma)$, so
  $i \geq \delta^-(\Gamma) \ge \eta n$.  Then, 
$$0 > \mathbf{y}\cdot \mathbf{1} = 
\sum_{j = 1}^{i} y_j + \sum_{j=i+1}^{n}y_j \ge
i y_i + (n - i)y_n \ge
\eta n y_i + (1 - \eta)n y_n = n \mathbf{y} \cdot \chi_{\eta,1-\eta}(\dir{v_iv_n}) \geq 0,$$
a contradiction.

The second statement is implied by basic linear programming theory,
if we take the perfect fractional $(\eta, 1-\eta)$-matching to be one with the smallest possible number of non-zero weights,
as then $w$ is a basic feasible solution.
\end{proof}

Note that if a directed graph $\Gamma$ admits a perfect $(\eta, \xi)$-weighted fractional matching~$w$ with $\eta \le \xi$ and $\eta + \xi = 1$, 
then $\alpha(\Gamma) \le \xi n$, because for every independent set $A$ in $\Gamma$ we have
\begin{align*}
|A| &= 
\sum_{a \in A} \left(\sum_{b \in N^+(a)} \eta w_{\dir{ab}}  + \sum_{b \in N^-(a)} \xi w_{\dir{ba}} \right) 
\leq \xi \sum_{a \in A} \left(\sum_{b \in N^+(a)} w_{\dir{ab}}  + \sum_{b \in N^-(a)} w_{\dir{ba}} \right) \leq \xi W 
\leq \xi n,
\end{align*}
where the initial equality holds since we have equality in~\eqref{eq:fracmatch}, and the penultimate inequality holds because (since $A$ is an independent set) every edge of $\Gamma$ contributes at most once to the sum. This shows that the minimum indegree condition of Lemma~\ref{wfracm} is best possible for $\eta \le 1/2$, since weaker conditions do not preclude the existence of independent sets of size greater than $(1 - \eta) n$.

\section{Triangle-tilings in regular pairs and triples}\label{sec-main}

Loosely speaking, the proof of Theorem~\ref{main} proceeds by iteratively constructing a triangle-tiling in $G$ which covers all of the vertices outside of a small `core' subset of vertices but leaves most vertices inside this `core' uncovered.  This gives a perfect triangle-tiling in~$G$, because the `core' is robust in the sense that it has a perfect triangle-tiling after the removal of any sufficiently small set of vertices (provided that the number of vertices remaining is divisible by $3$). Depending on the structure of the graph $G$, this `core' will either consist of sets $A$ and $B$ which form a super-regular pair with density greater than $\tfrac{1}{2}$, or of sets $A$, $B$ and $C$ which form three super-regular pairs each with density bounded below by a small constant.

We begin with the case where the `core' consists of a super-regular pair of density greater than $\tfrac{1}{2}$ (part (c) of Lemma~\ref{tiling}).  Let $G$ be a graph whose vertex set is the disjoint union of sets $A$ and $B$. Recall that a triangle $T$ in $G$ is an $A$-triangle if $T$ contains two vertices of $A$ and one vertex of $B$, and likewise that $T$ is a $B$-triangle if $T$ contains two vertices of $B$ and one vertex of $A$.

\begin{lemma} \label{tiling}
Suppose that $1/n \ll \gamma \ll \eps \ll \phi, \eps' \ll d \ll \omega$. Let $A$ and $B$ be disjoint sets of vertices with $n/3 + \omega n \leq |A|, |B| \leq 2n/3 - \omega n$ and $|A \cup B| = n$, and let $G$ be a graph on vertex set $V := A \cup B$ with $\alpha(G) \leq \gamma n$. Then the following statements hold.
\begin{enumerate}[label=(\alph*), noitemsep]
\item If $G[A, B]$ is $(\geq \!\! d, \eps)$-regular then $G$ admits a triangle-tiling covering all but at most $2 \eps n$ vertices of $G$. Moreover, for every $a$ and $b$ with $2a+b \leq |A| - \eps n$ and $a+2b \leq |B|-\eps n$ there is a triangle-tiling in $G$ which consists of $a$ $A$-triangles and $b$ $B$-triangles.
\item If $G[A, B]$ is $(d, \eps)$-super-regular then, provided $|A \sm S| + |B| + \floor{\phi \eps' n}$ is divisible by $3$, for every $S \subseteq A$ of size $|S| = \phi n$ there is a triangle-tiling in $G$ which covers every vertex of $G[V \sm S]$ and which covers precisely $\floor{\phi \eps' n}$ vertices of $S$. 
\item If $n$ is divisible by $3$ and $G[A, B]$ is $(1/2 + d, \eps)$-super-regular then $G$ contains a perfect triangle-tiling.
\end{enumerate}
\end{lemma}

\begin{proof}
  For (a) the triangles may be chosen greedily. Indeed, suppose that we have already chosen a triangle-tiling $\T$ consisting of at most $a$ $A$-triangles and at most $b$ $B$-triangles, then $\T$ covers at most $2a + b$ vertices of $A$, and at most $a+2b$ vertices of $B$. Taking $A' = A \sm V(\T)$ and $B' = B \sm V(\T)$, we find that $|A'|, |B'| \geq \eps n$. Since $G[A, B]$ is $(\geq \!\! d, \eps)$-regular it follows that $d_G(A', B') \geq d - \eps$, therefore some vertex $x \in A'$ has at least $(d-\eps)|B'| \geq (d-\eps)\eps n > \gamma n$ neighbours in $B'$. Since $\alpha(G) \leq \gamma n$ it follows that some two of these neighbours are adjacent, giving a $B$-triangle which can be added to $\T$. The same argument with the roles of $A'$ and $B'$ reversed yields instead an $A$-triangle which may be added to $\T$. This proves the second statement of~(a); the first follows by setting $a = \tfrac{1}{3}(2|A| - |B| - \eps n)$ and $b = \tfrac{1}{3}(2|B| - |A| - \eps n)$.

Next, for (b), 
let $z := \floor{\phi \eps' n}$, 
$t_4 := \floor{z/2}$ and $z' := z - 2 t_4 \in \{0, 1\}$,
so we will construct a triangle-tiling that covers all of $(A \sm S) \cup B$ and
exactly $z = 2t_4 + z'$ vertices of $S$.
Let $B_1' \subseteq B$ consist of all vertices in $B$ with fewer than $(d-\tfrac{\eps}{\phi}) |S|$ neighbours in $S$; since $G[S, B]$ is $(\geq \!\! d, \tfrac{\eps}{\phi})$-regular we have $|B_1'| \leq \tfrac{\eps}{\phi} n$. 
Form $B_1$ by adding at most $2$ arbitrarily selected vertices from $B \sm B_1'$ to $B_1'$ 
so that $|B \sm B_1| - t_4$ is divisible by $3$.
Since $G[A, B]$ is $(d, \eps)$-super-regular, every vertex of $B_1$ has at least $(d-\eps)|A| - |S| \geq \tfrac{dn}{3} > 2|B_1| + \gamma n$ neighbours in $A \sm S$. Since $\alpha(G) \leq \gamma n$, 
we may greedily form a triangle-tiling $\T_1$ of $A$-triangles in $G$ of size $|B_1|$ which covers every vertex of $B_1$ and
does not use any vertex from $S$. 
We now select uniformly at random a subset $B_2 \subseteq B \sm B_1$ of size $|B_2| = t_4$. 
Since every vertex in $A$ has at least $(d-\eps)|B| - |B_1| \geq \tfrac{dn}{3}$ neighbours in $B \sm B_1$, Theorem~\ref{chernoff} implies that, with probability $1-o(1)$, every vertex of $A$ has at least $\tfrac{\phi\eps' d}{7}n$ neighbours in $B_2$. Fix a choice of $B_2$ for which this event occurs. 
Let $S'$ be an arbitrarily selected subset of $S$ of size $z'$ (so $S'$ is either empty or a singleton) and
let $A' := (A \sm (S \cup V(\T_1))) \cup S'$
and $B' := B \sm (B_1 \cup B_2)$.
Recall that, by assumption, $|A \sm S| + |B| + z$ is divisible by $3$, so
\begin{equation*}
|A'| + |B'| = |A \sm S| + z' + |B| - |B_2| - |V(\T_1)|  = 
\left(|A \sm S| + |B| + z \right) - \left(3t_4 + |V(\T_1)|\right)
\end{equation*}
is divisible by $3$. Since $|B'|$ is divisible by $3$ by our selection of $B_1$ and $B_2$, it follows that $|A'|$ is divisible by $3$ as well.
Let $t_3 = \floor{\tfrac{\phi\eps' d}{15}n}$, 
$a := \tfrac{2}{3}|A'|-\tfrac{1}{3}|B'|$ and $b := \tfrac{2}{3}|B'|-\tfrac{1}{3}|A'| - t_3$. 
Since $G[A', B']$ is $(\geq \!\! d, \tfrac{\eps}{2})$-regular, (a) implies that there is a triangle-tiling $\T_2$ in $G[A' \cup B']$ such that $A'' := A' \sm V(\T_2)$ and $B'' := B' \sm V(\T_2)$ have sizes precisely $|A''| = |A'| - (2a+b) = t_3$ and $|B''| = |B'| - (a+2b) = 2t_3$. Since by the choice of $B_2$ each vertex of $A''$ has at least $\tfrac{\phi \eps' d}{7}n > 2|A''| + \gamma n$ neighbours in $B_2$, we may greedily form a triangle-tiling $\T_3$ in $G[A'' \cup B_2]$ consisting of exactly $t_3$ $B$-triangles which covers every vertex of $A''$ and which covers precisely $2t_3$ vertices of $B_2$. At this point we have obtained a triangle-tiling $\T_1 \cup \T_2 \cup T_3$ in $G$ which covers every vertex of $A$ except for those in $S \sm S'$ and every vertex of $B$ except for the precisely $2 t_3$ vertices in $B''$ and the precisely $t_4 - 2 t_3$ vertices in $B_2 \sm V(\T_3)$. Therefore, in total, precisely $t_4$ vertices of $B$ remain uncovered, each of which has at least $(d-\tfrac{\eps}{\phi}) |S| - |S'| > 2|B_2| + \gamma n$ neighbours in $S \sm S'$ by the choice of $B_1$. We may therefore greedily form a triangle-tiling $\T_4$ of $A$-triangles in $G$ which covers all the remaining uncovered vertices in $B$ and precisely $2t_4$ vertices of $S \sm S'$. Then $\T_1 \cup \T_2 \cup \T_3 \cup \T_4$ is the claimed triangle-tiling.

Finally, since none of the assumptions for (c) involve $\phi$ or $\eps'$, we may assume that $\phi \ll \eps'$. We also assume without loss of generality that $|B| \ge |A|$.  Since $\alpha(G) \leq \gamma n$, we may greedily form a matching $M$ of size at least $(|B|-\gamma n)/2 \geq n/10$ in $G[B]$. Fix such a matching $M$, and form an auxiliary bipartite graph $H$ with vertex classes $A$ and $M$ where $a \in A$ and $e = xy \in M$ are adjacent if and only if $xyz$ is a triangle in $G$. Note that for every edge $e = xy \in M$ we have that 
$$\deg_H(e) = |N_G(x) \cap N_G(y) \cap A| \geq 2 ((1/2 + d) - \eps)|A| - |A| \geq d |A|,$$ so $H$ has density at least $d$. By Lemma~\ref{absorber}, applied to $H$ with $\eps'$ here in place of $\eps$ there, we may choose subsets $X \subseteq A$ and $M' \subseteq M$ such that $|X| = \phi n$, $|M'| = (1-\eps)\phi n$ and such that $H[X', M']$ contains a perfect matching for every subset $X' \subseteq X$ with $|X'| = |M'|$. Let $B' := B \sm V(M')$ and $n' := |A| \cup |B'|$. 
Then, since we assumed that $|B| \ge |A|$, we have $n'/3 + \omega n' \leq |A|, |B'| \leq 2n'/3 - \omega n'$, so we can apply (b) to $G[A \cup B']$ with $A$, $B'$ and $X$ in place of $A$, $B$ and $S$ respectively to obtain a triangle-tiling $\T_1$ in $G$ which covers every vertex of $G$ except for the vertices of $V(M')$ and precisely $(1-\eps')\phi n$ vertices of $X$. So, taking $X'$ to be the vertices of $X$ not covered by $\T_1$, we have $|X'| = |M'|$. By the choice of $X$ and $M'$ it follows that $H[X', M']$ contains a perfect matching, which corresponds to a perfect triangle-tiling $\T_2$ in $G[X' \cup V(M')]$. This gives a perfect triangle-tiling $\T_1 \cup \T_2$ in $G$.
\end{proof}

We now turn to the case where the `core' consists of three sets which form three super-regular pairs, for which the following lemma is analogous to Lemma~\ref{tiling}.

\begin{lemma} \label{tripartitetiling}
Suppose that $1/n \ll \gamma, \eps \ll d, \omega$, and that $3$ divides $n$. Let $V_1, V_2$ and $V_3$ be disjoint sets of vertices with $|V_i| \geq n/6 + \omega n$ for each $i \in [3]$ such that $V := \bigcup_{i \in [3]} V_i$ has size $|V| = n$. Let $G$ be a graph on vertex set $V$ with $\alpha(G) \leq \gamma n$ such that $G[V_i, V_j]$ is $(d, \eps)$-super-regular for each distinct $i, j \in [3]$. Then $G$ contains a perfect triangle-tiling.
\end{lemma}

To prove Lemma~\ref{tripartitetiling} we use the celebrated Blow-up Lemma of Koml\'os, S\'ark\"ozy and Szemer\'edi~\cite{KSSz97} to obtain a perfect triangle-tiling. For simplicity, we state this only in the (very) special case that we use. Note that our definition of super-regularity differs slightly from theirs, but it is not hard to show that the two definitions are equivalent up to some modification of the constants involved (see, for example,~\cite[Fact 2]{RR}), so the validity of Theorem~\ref{blowup} is unaffected.

\begin{theorem} [Blow-up Lemma for triangle-tilings] \label{blowup}
Suppose that $1/n \ll \eps \ll d$. 
Let $A, B$ and $C$ be disjoint sets of vertices with $|A| = |B| = |C| = n$, and let $G$ be a graph on vertex set $V := A \cup B \cup C$ such that $G[A, B]$, $G[B, C]$ and $G[C, A]$ are each $(d, \eps)$-super-regular. Then $G$ contains a perfect triangle-tiling.
\end{theorem}

The proof of Lemma~\ref{tripartitetiling} proceeds by iteratively deleting triangles from $G$ with two vertices in one cluster and one in another cluster, until the same number of vertices remain in each cluster. We complete the proof by applying the Blow-up Lemma to obtain a perfect triangle-tiling covering all remaining vertices.

\begin{proof}[Proof of Lemma~\ref{tripartitetiling}]
Throughout this proof we perform addition on subscripts modulo~$3$. For each $i \in [3]$, the fact that $G[V_i, V_{i+1}]$ is $(d, \eps)$-super-regular implies that each vertex $v \in V_i$ has $|N(v) \cap V_{i+1}| \geq (d-\eps)|V_{i+1}| \geq dn/6$. So if we choose uniformly at random a set $Z_j \subseteq V_j$ of size $\omega n$ for each $j \in [3]$, then $|N(v) \cap Z_{i+1}|$ is hypergeometrically distributed with expectation at least $d \omega n/6$. By Theorem~\ref{chernoff} the probability that $v$ has fewer than $d \omega n/7$ neighbours in $|Z_{i+1}|$ declines exponentially with $n$, and likewise the same is true of the probability that $v$ has fewer than $d \omega n/7$ neighbours in $|Z_{i+2}|$. Taking a union bound, with positive probability it holds that for each $i \in [3]$ every vertex $v \in V_i$ has at least $d \omega n/7$ neighbours in each of $Z_{i+1}$ and $Z_{i+2}$. We fix such an outcome of our random selection of the sets $Z_j$, and define $X_i^0 = V_i \sm Z_i$ for each $i \in [3]$. Without loss of generality we may assume that $\frac{n}{6} \leq |X_1^0| \leq |X_2^0| \leq |X_3^0| \leq \frac{2n}{3} - 3\omega n$. 

We now proceed by an iterative process. At time step $t \geq 0$, if we have $|X_1^t| = |X_2^t| = |X_3^t|$ then we terminate. Otherwise, we choose a triangle $xyz$ in $G$ with $x \in X_2^t$ and $y, z \in X_3^t$ (we shall explain shortly why this will always be possible). We then set $Y_j^{t+1} := X_j^t \sm \{x, y, z\}$ for $j \in [3]$ and define $X_1^{t+1}, X_2^{t+1}$ and $X_3^{t+1}$ such that $\{X_1^{t+1}, X_2^{t+1}, X_3^{t+1}\} = \{Y_1^{t+1}, Y_2^{t+1}, Y_3^{t+1}\}$ and $|X_1^{t+1}| \leq |X_2^{t+1}| \leq |X_3^{t+1}|$, before proceeding to the next time step~$t+1$.

Suppose that this procedure does not terminate prior to some time step $T$. Using the fact that $3$ divides $n$ it is easily checked that we must then have $|X_3^{t+2}| - |X_1^{t+2}| \leq |X_3^{t}| - |X_1^{t}| - 3$ for each $t \in [T-2]$. In other words, the size difference between the smallest and largest set decreases by at least 3 over each two time steps. Similarly we find that $|X_1^{t}| - |X_1^{t+2}| \leq 1$ for each $t \in [T-2]$, meaning that the smallest set size decreases by at most 1 over each two time steps. Furthermore, if at some time $t$ we have $0 < |X_3^{t}| - |X_1^{t}| < 3$, then (since $3$ divides $n$) we must have $|X_1^t|+2 = |X_2^t| +1 = |X_3^t|$, whereupon the procedure will terminate at time $t+1$. It follows that the procedure must terminate at some time $T$, and moreover that 
$$T \leq \frac{2}{3}\left(|X_3^0| - |X_1^0|\right) \leq \frac{2}{3}\left(\left(\frac{2n}{3} - 3 \omega n\right) - \frac{n}{6}\right) = \frac{n}{3} - 2 \omega n.$$
This implies that at each time $t < T$ we have $|X_3^t| \geq |X_2^t| \geq |X_1^t| \geq |X_1^0| - \ceiling{\frac{t}{2}} \geq |X_1^0| - \frac{T}{2} \geq \omega n$, and so throughout the procedure it is always possible to pick a triangle as desired. Indeed, $G[X_2^t, X_3^t]$ is $(\geq\!\!d, \eps/\omega)$-regular by the Slicing Lemma~(Lemma~\ref{slicing}), so some vertex of $X_2^t$ has at least $(d-\eps/\omega)|X_3^t| \geq \omega d n/2$ neighbours in $X_3^t$. Since $\alpha(G) \leq \gamma n < \omega dn/2$ some two of these neighbours must be adjacent, giving the desired triangle. 

After the procedure terminates, define $V'_i := X_i^T \cup Z_i$ for each $i \in [3]$. Then $|V'_1| = |V'_2| = |V'_3| \ge 2 \omega n$, so by Lemma~\ref{slicing} and our choice of the sets $Z_j$ it follows that $G[V'_i, V'_j]$ is $(d \omega/7, \eps/2\omega)$-super-regular for each distinct $i, j \in [3]$. By Theorem~\ref{blowup}
there is a perfect triangle-tiling in $G[\bigcup_{i \in [3]} V_i']$; together with the triangles selected by the iterative procedure this gives a perfect triangle-tiling in $G$.
\end{proof}

\section{Proof of Theorem~\ref{main}}

In this section we prove Theorem~\ref{main}. The following lemma is the central part of the proof, showing that if a graph~$G$ can be decomposed into clusters which form regular and super-regular pairs, indexed by a graph~$R$ which admits a bounded degree spanning tree, then by `working inwards' from the leaves of the tree we can form a perfect triangle-tiling in~$G$.

\begin{lemma} \label{treeargument}
Suppose that $1/m \ll \gamma \ll 1/k \ll \eps \ll d, \omega$.
Let~$G$ be a graph whose vertex set is partitioned into~$k$ sets $V_1, \dots, V_k$, and let~$R$ be a graph with vertex set~$[k]$ which admits a spanning tree~$T$ of maximum degree at most 10. Suppose also that the following statements hold.
\begin{enumerate}[label=(\alph*), noitemsep]
\item \label{tree-a} $|V_1| \ge (1-\eps) m$. 
\item \label{tree-b} $V_1$ admits either a partition into parts~$A_1$ and~$B_1$ with $|A_1|, |B_1| \geq (1/3+\omega)|V_1|$ such that $G[A_1, B_1]$ is $(1/2+d, \eps)$-super-regular, or a partition into parts $A_1, B_1$ and~$C_1$ with $|A_1|, |B_1|, |C_1| \geq (1/6+\omega)|V_1|$ such that $G[A_1, B_1]$, $G[A_1, C_1]$ and $G[B_1, C_1]$ are each $(d, \eps)$-super-regular. 
\item \label{tree-c} For each $2 \leq i \leq k$, $(1 - \eps)m \le |V_i| \le m$ and~$V_i$ admits a partition into parts~$A_i$ and~$B_i$ with $|A_i|, |B_i| \geq (1/3+\omega)m$ such that $G[A_i, B_i]$ is $(d, \eps)$-super-regular. 
\item \label{tree-d} If $ij \in E(R)$, then at least~$m/5$ vertices of~$V_i$ have at least~$dm/5$ neighbours in~$V_j$.
\item \label{tree-e} $\alpha(G) \leq \gamma m$.
\end{enumerate}
Then~$G$ contains a triangle-tiling covering all but at most two vertices of $G$.
\end{lemma}

\begin{proof}
Introduce new constants~$\phi$ and~$\eps'$ with $\eps \ll \phi \ll \eps' \ll d$ and iterate the following process. Pick a leaf of~$T$ other than vertex~$1$, say vertex~$i$, and let~$j$ be the neighbour of~$i$ in~$T$. We will show that there exists a triangle-tiling in $G[V_i \cup V_j]$ that covers every vertex of~$V_i$ and at most~$2 \phi m$ vertices of~$V_j$.
  We then delete the vertices covered by this tiling from~$G$ and delete vertex~$i$ from~$T$. We proceed in this way until only vertex~$1$ of~$T$ remains. We then arbitrarily delete at most two further vertices of~$V_1$ so that the number of remaining vertices in $V_1$ is divisible by three. Since, at this point, we have removed 
  at most $2 \phi m \cdot \Delta(T) + 2 \le 21 \phi m \leq \eps' m/7$ vertices from~$V_1$, by~\ref{tree-a},~\ref{tree-b} and~\ref{tree-e} there exists a bipartition or tripartition of the remaining vertices of~$V_1$ which satisfies the conditions of Lemma~\ref{tiling}(c) or Lemma~\ref{tripartitetiling} respectively (with~$\omega/2$,~$\eps'$ and~$2\gamma$ in place of~$\omega$,~$\eps$ and~$\gamma$ respectively).
In either case there is a perfect triangle-tiling in the graph induced by the remaining vertices of~$V_1$, which together with the deleted triangle-tilings gives a triangle-tiling in~$G$ covering every vertex except for the at most two deleted vertices.

It therefore suffices to show that we can find the desired triangle-tiling in $G[V_i \cup V_j]$ at each step of this process.
To this end, let~$S'$ be the set of vertices of~$V_i$ which have at least~$dm/6$ neighbours in~$V_j$. Observe that previous deletions can have removed at most $2 \phi m \cdot \Delta(T) \leq dm/30$ vertices from each of~$V_i$ and~$V_j$, so by~\ref{tree-d} we have~$|S'| \geq m/6$, and by~\ref{tree-c} the remaining vertices of~$V_i$ can be partitioned into parts~$A_i$ and~$B_i$ with $|A_i|, |B_i| \geq (1/3 + \omega/2)m$ such that $G[A_i, B_i]$ is $(d, \eps')$-super-regular. Without loss of generality we may assume that $|S' \cap A_i| \ge |S' \cap B_i|$, so $|S' \cap A_i| \ge |S'|/2 \ge m/12$ and we can arbitrarily select $S \subseteq S' \cap A_i$ of size $\phi n$.  Now we may use Lemma~\ref{tiling}(b) (again with~$\omega/2$,~$\eps'$ and~$2\gamma$ in place of~$\omega$,~$\eps$ and $\gamma$ respectively) to find a triangle-tiling~$\T$ in~$G[V_i]$ which covers every vertex of~$V_i \sm S$. Since each uncovered vertex has at least $dm/6 \geq 2 \phi m + \gamma m$ neighbours in~$V_j$, we may greedily extend~$\T$ to a triangle-tiling~$\T'$ in~$G$ which covers every vertex of~$V_i$ and which covers at most~$2 \phi m$ vertices of~$V_j$.
\end{proof}

It now suffices to show that for every graph $G$ satisfying the conditions of Theorem~\ref{main}, we can delete triangles and/or vertices from $G$ to obtain a subgraph whose structure meets the conditions of Lemma~\ref{treeargument}. The following lemma shows how to do this under the additional assumption that $G$ has no large sparse cut; this assumption is useful as it allows us to assume that the reduced graph $R$ of $G$ is connected, and so has spanning trees of bounded maximum degree. For this we make the following definition: given a graph $G$ and a partition $\{A, B\}$ of $V(G)$, we say that an edge of $G$ is \emph{crossing} if it has one endvertex in $A$ and one endvertex in $B$.

\begin{lemma}\label{nosparsecut}
For every $\omega, \psi > 0$ there exist $n_0, \gamma > 0$ such that the following statement holds. Let $G$ be a graph on $n \geq n_0$ vertices with $\delta(G) \geq n/3 + \omega n$ and $\alpha(G) \leq \gamma n$. Suppose additionally that for every partition $\{A, B\}$ of $V(G)$ with $|A|, |B| \geq n/3$ there are at least $\psi n^2$ crossing edges of $G$. Then $G$ contains a triangle-tiling covering all but at most two vertices of $G$ (so in particular, if 3 divides $n$ then $G$ contains a perfect triangle-tiling).
\end{lemma}

\begin{proof}
Introduce new constants satisfying the following hierarchy:
  \begin{equation*}
    1/n \ll \gamma \ll 1/D \ll 1/T \ll 1/t \ll \eps' \ll \eps \ll d \ll \omega, \psi.
  \end{equation*}
Then we may assume that $n$ and $T$ are large enough to apply Theorem~\ref{reglem} with constants $\eps'/2, d, t$ and $q=1$. 
We also assume without loss of generality that $\omega^{-1}$ is an integer, and define $D' := 30\omega^{-1}(D!)$. 
Let $G$ be as in the statement of the lemma, and apply Theorem~\ref{reglem} to $G$ to obtain a spanning subgraph $G' \subseteq G$, an integer $k'$ with $t \leq k' \leq T$, an exceptional set $U_0$ of size at most $\eps' n/2$ and clusters $U_1, \dots, U_{k'}$ of equal size. We now remove at most $D'$ vertices from each cluster so that the number of remaining vertices in each cluster is divisible by $D'$, and add all removed vertices to the exceptional set $U_0$. Since the total number of vertices moved in this way is at most $D'k' \leq 30\omega^{-1}(D!)T \leq \eps' n/2$, and at most $D' \leq \eps' n/2T \leq \eps'/2 |U_i|$ vertices are removed from each cluster $U_i$, by Lemma~\ref{slicing} the resulting partition of $V(G)$ into $U_0, U_1, \dots, U_{k'}$ has the following properties.
  \begin{enumerate}[label=(\roman*), noitemsep]
    \item \label{pfi} $|U_0| \leq \eps' n$ and $|U_1| = |U_2| = \ldots = |U_{k'}| =: m'$, where $D'$ divides $m'$.
    \item \label{pfii} $d_{G'}(v) \geq d_G(v) - (\eps' + d)n \geq n/3 + 2\omega n/3$ for all $v \in V(G)$. 
    \item \label{pfiii} $e(G'[U_i]) = 0$ for all $i \in [k']$. 
    \item \label{pfiv} for each distinct $i, j \in [k']$ either $G'[U_i, U_j]$ is $(\geq\!\!d, \eps')$-regular or $G'[U_i, U_j]$ is empty.
    \newcounter{save} 
    \setcounter{save}{\value{enumi}}
  \end{enumerate}
In particular~\ref{pfi} implies that $(1-\eps')n/k' \leq m' \leq n/k'$.   
We form the reduced graph $R$ on vertex set $[k']$ in the usual way, that is, with $ij \in E(R)$ if and only if $e(G'[U_i, U_j]) > 0$. For each $i \in [k']$ the number of edges of $G'$ with an endvertex in $U_i$ is at least $m'(n/3+2\omega n/3)$ by~\ref{pfii}. Also, by~\ref{pfiii} there is no edge in $G'[U_i]$, and by~\ref{pfi} there are at most at most $m'\eps' n$ edges in $G'[U_0, U_i]$. Since for each $j \in [k']$ there are at most $(m')^2$ edges in $G'[U_i, U_j]$, it follows that 
\begin{equation}\label{eq:mindeg}
\delta(R) \geq \frac{m'(n/3+2\omega n/3) - m'\eps' n}{(m')^2} 
\geq \left(\frac{1}{3} + \frac{2\omega}{3}-\eps'\right) \frac{n}{m'}
\geq \left(\frac{1}{3} + \frac{\omega}{2}\right)k'.
\end{equation} 
Now consider a partition $\{A_R, B_R\}$ of $[k']$ with $|A_R|, |B_R| \geq \delta(R)$, and define $A := U_0 \cup \bigcup_{i \in A_R} U_i$ and $B := \bigcup_{i \in [k'] \sm B_R} U_i$. Then 
$$|A|, |B| \geq \delta(R) m' \geq \left(\frac{1}{3} + \frac{\omega }{2}\right)k' \cdot \frac{(1-\eps')n}{k'} \geq\frac{n}{3},$$
so by assumption $G$ has at least $\psi n^2$ crossing edges. By~\ref{pfii} at most $(d+\eps')n^2$ edges of $G$ are not in $G'$, and by~\ref{pfi} at most $\eps' n^2$ edges of $G$ intersect $U_0$, so $G'$ contains at least $\psi n^2 - (d+\eps') n^2 - \eps' n^2 > 0$ crossing edges which do not intersect $U_0$. Let $U_i$ and $U_j$ be clusters containing the endvertices of some such edge; then $ij$ is a crossing edge of $R$. In other words, for every partition $\{A_R, B_R\}$ of $[k']$ with $|A_R|, |B_R| \geq \delta(R)$ there is a crossing edge of $R$. Since every connected component of $R$ has size at least $\delta(R)$, it follows that $R$ is connected.
 
We now form a set $V_1$ from which we shall form the `core' set of vertices mentioned in the proof overview at the beginning of Section~\ref{sec-main}. Suppose first that there exist $i, j \in [k']$ with $d(G'[U_i, U_j]) \geq 2/3$. Then $G'[U_i, U_j]$ is $(\geq\!\!3/5, \eps')$-regular by~\ref{pfiv}. In this case we define $V_1 := U_i \cup U_j$, and for convenience of notation later we define $X_1 := U_i$ and $Y_1 := U_j$. 
Now suppose instead that $d(G'[U_i, U_j]) < 2/3$ for every $i, j \in [k']$, that is, that each $G'[U_i, U_j]$ has at most $2(m')^2/3$ edges. Then we have an extra factor of $2/3$ in the denominator of the second term of~\eqref{eq:mindeg}, so we have $\delta(R) \geq k'/2$, and so $R$ contains a triangle $ij\ell$ by Mantel's theorem. 
In this case we take $V_1 := U_i \cup U_j \cup U_\ell$ and set $X_1 := U_i$, $Y_1 := U_j$, and $Z_1 := U_\ell$.
We define an auxiliary graph $R_0$ to be the subgraph of $R$ formed by deleting vertices $i$ and $j$ in the former case, and by deleting vertices $i, j$ and $\ell$ in the latter case. 

Since $\omega^{-1}$ is an integer, we may write 
$\eta := 1/3 + \omega/10$ as a rational number with 
denominator $L := 30 \cdot \omega^{-1}$. Let $\dir{R_0}$ be the directed graph formed from $R_0$ by replacing each edge by a pair of edges, one in each direction. 
Then by Lemma~\ref{wfracm}, we can find a perfect $(\eta, 1-\eta)$-weighted 
fractional matching $w$ in $\dir{R_0}$ in which all weights are rational, and the least common denominator $L'$ of all weights is bounded above by a function of $|V(R_0)|$ and $L$, that is, a function of $k'$ and $\omega$. Since $k' \leq 1/T$ and we assumed that $1/D \ll 1/T, \omega$, we may assume that $L' \leq D$, so $L'$ divides $D!$, and so $D!w_{\dir{ij}}$ is an integer for every $\dir{ij} \in \dir{R_0}$. Define $m := m'/D!$, and observe that that since $D' = D!L$ divides $m$ by~\ref{pfi}, both $m$ and $\eta m$ are integers.

We now partition each cluster not contained in $V_1$ into parts of size $\eta m$ and $(1 - \eta) m$ according to the weights in $w$, using the following probabilistic argument.
For every $i \in V(R_0)$, we select 
a partition $\U_i$ of $U_i$ uniformly at random from all such partitions in which
exactly $\sum_{j \in N^+(i)} D!w_{\dir{ij}}$ sets are of size $\eta m$ and 
exactly $\sum_{j \in N^-(i)} D!w_{\dir{ji}}$ sets are of size $(1 - \eta) m$.
Since $w$ is a perfect fractional $(\eta, 1-\eta)$-weighted matching, by~\eqref{eq:fracmatch} we have 
$$\eta m \sum_{j \in N^+(i)} D!w_{\dir{ij}} + (1-\eta)m \sum_{j \in N^-(i)} D!w_{\dir{ji}} = D!m = m' = |U_i|,$$
so we can indeed partition $U_i$ in this way. We also consider the two or three clusters contained in $V_1$ to be partitioned into a single part. That is, for each $i \in [k'] \sm V(R_0)$ we set $\U_i$ to be the trivial partition $\{U_i\}$ of $U_i$. 
Now consider any edge $ij \in E(R)$, and recall that 
$G'[U_i, U_j]$ is $(\geq \!\! d, \eps')$-regular by~\ref{pfiv}, so 
by Lemma~\ref{random-slicing}\footnote{Note that $m$ is much smaller than $\eps' m'$ (since $D$ is much larger than $1/\eps'$) so we must use the random slicing lemma (Lemma~\ref{random-slicing}) here, as opposed to, say, the standard slicing lemma (Lemma~\ref{slicing}).}, with probability at least $1 - e^{-\Omega(n)}$
we have that $G'[U'_i, U'_j]$ is $(\geq \!\! d, \eps)$-regular 
for every $U'_i \in \U_i$ and 
for every $U'_j \in \U_j$.
Taking a union bound over all of the at most $\binom{k'}{2}$ edges of $R$ we find that with positive probability this property holds for every edge of $R$. Fix a choice of partitions $\U_i$ for $i \in [k']$ for which this is the case.

We now define another auxiliary graph $R_1$ with vertex set $\bigcup_{i \in [k']} \U_i$ in which, for each distinct $i, j \in [k']$, each $X \in \U_i$ and each $Y \in \U_j$, there is an edge $XY$ if and only if $G'[X, Y]$ is $(\geq \!\! d, \eps)$-regular. Observe that by our choice of partitions $\U_i$ the graph $R_1$ is then a blow-up of $R$, formed by replacing each vertex $i \in [k']$ by a set of $|\U_i|$ vertices and replacing each edge $ij \in E(R)$ by a complete bipartite graph between the corresponding sets. In particular, $R_1$ is connected. Also note that for each distinct $i, j \in [k']$ with $ij \notin E(R)$, each $X \in \U_i$ and each $Y \in \U_j$, the graph $G'[X, Y]$ is empty by~\ref{pfiv}. 

Next, for every edge $\dir{ij} \in E(\dir{R_0})$, we define $s_{ij} := D! \cdot w_{\dir{ij}}$. We then label $s_{ij}$ of the sets
in $\U_i$ of size $\eta m$ as $X_{ij}^1, \dots, X_{ij}^{s_{ij}}$ and label $s_{ij}$ 
of the sets in $\U_j$ of size $(1 - \eta)m$ as 
$Y_{ij}^1, \dots, Y_{ij}^{s_{ij}}$. Since $\U_i$ has exactly $\sum_{j \in N^+(i)} s_{{ij}}$ sets of size $\eta m$ and 
exactly $\sum_{j \in N^-(i)} s_{{ji}}$ sets of size $(1 - \eta) m$, we may do this so that for each $i \in [k']$ each set in $\U_i$ is uniquely labelled. We now relabel the sets $X_{ij}^\ell$ and $Y_{ij}^\ell$ for $\dir{ij} \in E(\dir{R_0})$ and $\ell \in s_{ij}$ as $X_2, \dots, X_{k}$ and $Y_2, \dots, Y_{k}$ respectively, where $k -1 := \sum_{\dir{ij} \in E(\dir{R_0})} s_{ij} = D! \sum_{\dir{ij} \in E(\dir{R_0})} w_{ij} = D! |V(R_0)|$ since $w$ is perfect, so $k' \leq k \leq D! k'$. Then for each $2 \leq \ell \leq k$ our choice of partition implies that
$G'[X_\ell, Y_\ell]$ is $(\geq \!\! d, \eps)$-regular; we define $V_\ell := X_\ell \cup Y_\ell$, and observe that $|V_\ell| = m$.

We now define a final auxiliary graph $R^*$ with vertex set $[k]$ in which $ij$ is an edge of $R^*$ if and only if $e(G'[V_i, V_j]) > 0$. Observe that $R^*$ is then a contraction of $R_1$, in which the vertices of $R_1$ corresponding to the sets $X_1$ and $Y_1$ (and $Z_1$ if defined) are contracted to the single vertex $1$ of $R^*$, and for $2 \leq i \leq k$ the vertices of $R_1$ corresponding to $X_i$ and $Y_i$ are contracted to the single vertex $i$ of $R^*$. So, since $R_1$ is connected, $R^*$ is connected also. 
Now suppose that $ij$ is an edge of $R^*$. Since $G'[V_i, V_j]$ is nonempty there must exist sets $S \in \{X_i, Y_i, Z_i\}$ and $T \in \{X_j, Y_j, Z_j\}$ such that $G'[S, T]$ is non-empty (ignore $Z_i$ unless $i=1$ and $Z_1$ exists, and likewise for $Z_j$). We then have $S \in \U_{i'}$ and $T \in \U_{j'}$ for some $i', j' \in [k']$, so $ST$ is an edge of $R_1$, and so $G'[S, T]$ is $(\geq \!\! d, \eps)$-regular. Also, a similar calculation to~\eqref{eq:mindeg} shows that we must have $\delta(R^*) \geq k/3$, so by Theorem~\ref{spantree} there is a spanning tree $T$ in $R^*$ with $\Delta(T) \leq 3$. 

To recap, at this point we have a formed a partition $\{U_0, V_1, \dots, V_k\}$ of $V(G)$ and a graph $R^*$ with vertex set $[k]$ which contains a spanning tree of maximum degree at most~$3$, such that the following statements hold. 

\begin{enumerate}[label=(\roman*), noitemsep]
\setcounter{enumi}{\value{save}}
\item \label{pfv} $V_1$ admits either a partition $\{X_1, Y_1\}$ with $|X_1|= |Y_1| = m'$ such that $G'[X_1, Y_1]$ is $(\geq \!\! 3/5, \eps')$-regular, or a partition $\{X_1, Y_1, Z_1\}$ with $|X_1| = |Y_1| = |Z_1| = m'$ such that $G'[X_1, Y_1]$, $G'[X_1, Z_1]$ and $G'[Y_1, Z_1]$ are each $(\geq \!\!d, \eps')$-regular. 
\item \label{pfvi} For each $2 \leq i \leq k$, we have $|V_i| = m$ and~$V_i$ admits a partition $\{X_i, Y_i\}$ with $|X_i|, |Y_i| \geq \eta m = (1/3+\omega/10)m$ such that $G'[X_i, Y_i]$ is $(\geq \!\!d, \eps)$-regular. 
\item \label{pfvii} If $ij \in E(R^*)$, then there are sets $S \subseteq V_i$ and $T \subseteq V_j$ with $|S| \geq |V_i|/3$ and $|T| \geq |V_j|/3$ such that $G'[S, T]$ is $(\geq \!\!d, \eps)$-regular.
\end{enumerate}

If we are in the first case of~\ref{pfv}, then by Lemma~\ref{makesuperreg} we may choose subsets $A_1 \subseteq X_1$ and $B_1 \subseteq Y_1$ with $|A_1|, |B_1| \geq (1-\eps')m'$ such that $G'([A_1, B_1])$ is $(3/5, 2\eps')$-super-regular, and we then define $W_1 := A_1 \cup B_1$. If we are instead in the second case, by three applications of Lemma~\ref{makesuperreg} we may choose subsets $A_1 \subseteq X_1$, $B_1 \subseteq Y_1$ and $C_1 \subseteq Z_1$ with $|A_1|, |B_1|, |C_1| \geq (1-2\eps')m'$ such that $G'([A_1, B_1])$, $G'([B_1, C_1])$ and $G'([C_1, A_1])$ are each $(d, 3\eps')$-super-regular, and we then define $W_1 := A_1 \cup B_1 \cup C_1$. Next, for each $2 \leq \ell \leq k$, by~\ref{pfvi} and Lemma~\ref{makesuperreg} we may choose subsets $A_\ell \subseteq X_\ell$ and $B_\ell \subseteq Y_\ell$ with $|A_\ell| \geq (1-\eps)|X_\ell|$ and $|B_\ell| \geq (1-\eps)|Y_\ell|$ such that $G'[A_\ell, B_\ell]$ is $(d, 2\eps)$-super-regular, and define $W_\ell := A_\ell \cup B_\ell$. Finally, define $W_0 := U_0 \cup \bigcup_{i \in [k]} V_i \sm W_i$. Then $\{W_0, W_1, \dots, W_k\}$ is a partition of $V(G)$ and, since $|W_i| \geq (1-\eps)|V_i|$ for each $i \in [k]$, we have $|W_0| \leq 2\eps n$. 

Write $W_0 := \{x_1, \dots, x_q\}$, so $q \leq 2\eps n$. To complete the proof we greedily form a triangle-tiling $\T = \{T_1, \dots, T_q\}$ such that $x_i \in T_i$ for each $i \in [q]$ and $|V(\T) \cap W_j| \leq 20 \eps |W_j|$ for each $j \in [k]$. To see that this is possible, suppose that we have already chosen triangles $T_1, \dots, T_{s-1}$ for some $s \in [q]$, let $X := \bigcup_{i \in [s-1]} V(T_i)$ be the set of vertices covered by these triangles, and let the set $X'$ consist of all vertices in sets $W_i$ with $|X \cap W_i| \geq 18 \eps |W_i|$ (that is, from which the previously-chosen triangles cover more than a $18 \eps$-proportion of the vertices). Then we have $18\eps |X'| \leq |X| \leq 3q \leq 6\eps n$, so $|X'| \leq n/3$, and so $x_s$ has at least $\delta(G) - |X| - |X'| - |W_0| \geq \omega n - 10\eps n \geq \omega n/2$ neighbours not in $X$, $X'$ or $W_0$, so (since $\alpha(G) \leq \gamma n < \omega n/2$) two of these neighbours must be adjacent, giving the desired triangle $T_s$ containing $x_s$. Having chosen $T_s$ in this way for every $s \in [q]$ to obtain $\T$, observe that since we chose each $T_s$ to avoid every set $W_i$ from which at least $18 \eps |W_i|$ vertices were covered by previously-chosen triangles, we must have $|V(\T) \cap W_i| \leq 20 \eps |W_i|$ for each $i \in [k]$, as desired.

Finally, for each $i \in [k]$ define $A_i' := A_i \sm V(\T)$, $B_i' := B_i \sm V(\T)$, $V'_i := W_i \sm V(\T)$. Also define $V' := V(G) \sm V(\T)$ and $H := G[V']$. We claim that the graphs $H$ and $R^*$ and the partition $\{V'_1 \dots, V'_k\}$ of $V(H)$ meet the properties of Lemma~\ref{treeargument} with $\eps^* := 200\eps$, $\omega' := \omega/20$ and $\gamma' := 2\gamma k'(D!)$ in place of $\eps$, $\omega$ and $\gamma$ respectively and with $m, d$ and $k$ playing the same role there as here. Indeed, our constant hierarchy allows us to assume that $1/m \ll \gamma' \ll 1/k \ll \eps^* \ll d \ll \omega'$, as required. Also observe that for each $i \in [k]$ we have $|V'_i| \geq |V_i| - 20 \eps |V_i| - \eps |V_i| = (1-21\eps)|V_i|$, so certainly $|V'_i| \geq (1-\eps^*)m$ for each $i \in [k]$. So Lemma~\ref{treeargument}\ref{tree-a} holds, and Lemma~\ref{treeargument}\ref{tree-b} and~\ref{tree-c} follow immediately from our choice of sets $A_\ell$ and $B_\ell$ (and possible $C_1$). Also, for each $ij \in E(R^*)$ by~\ref{pfvii} there exist sets $S \subseteq V'_i$ and $T \subseteq V'_j$ with $|S| \geq |V'_i|/4$ and $|T| \geq |V_j'|/4$ such that $G'[S, T]$ is $(\geq \!\!d, 2\eps)$-regular, which implies that at least $(1-2\eps)|S| \geq m/5$ vertices in $S$ have at least $(d-2\eps)|T| \geq dm/5$ neighbours in $T$, so Lemma~\ref{treeargument}\ref{tree-d} holds. Last of all, Lemma~\ref{treeargument}(e) holds since $\alpha(H) \leq \alpha(G) \leq \gamma n \leq \gamma (2k'm')= \gamma' m$. So we may apply Lemma~\ref{treeargument} to obtain a triangle-tiling covering all but at most two vertices of $H$; together with $\T$ this yields a triangle-tiling in $G$ covering all but at most two vertices.
\end{proof}

Finally, to complete the proof of Theorem~\ref{main} it remains only to consider graphs $G$ which admit a large sparse cut. In this case we show that can remove a small number of vertices to obtain two vertex-disjoint subgraphs $G_A$ and $G_B$ of $G$ whose vertex sets partition $V(G)$ and each of which satisfies a stronger minimum degree condition. We then apply Theorem~\ref{BMSthm} to obtain a perfect triangle-tiling in each of $G_A$ and $G_B$ (alternatively, one could note that the stronger minimum degree conditions preclude either $G_A$ or $G_B$ from having a large sparse cut and apply Lemma~\ref{nosparsecut}).

\begin{proof}[Proof of Theorem~\ref{main}]
Fix $\omega > 0$ and choose $n_0$ sufficiently large and $\gamma$ sufficiently small for Lemma~\ref{nosparsecut} with $\omega^2/40$ in place of $\psi$ and also so that we can apply Theorem~\ref{BMSthm} with $\omega/2$, $n_0/3$ and $3\gamma$ in place of $\omega$, $n_0$ and $\gamma$ respectively. Now let $G$ be a graph on $n \geq n_0$ vertices with $\delta(G) \geq n/3 + \omega n$ and $\alpha(G) \leq \gamma n$. If for every partition $\{A, B\}$ of $V(G)$ with $|A|, |B| \geq n/3$ there are at least $\omega^2 n^2/40$ crossing edges of $G$, then $G$ contains a triangle-tiling covering all but at most two vertices by Lemma~\ref{nosparsecut}, so we are done. So we may assume that some partition $\{A, B\}$ of $V(G)$ with $|A|, |B| \geq n/3$ has fewer than $\omega^2 n^2/40$ crossing edges. Fix such a partition with the smallest number of crossing edges. Note that we cannot have $|A| \leq n/3+1$, as then there would be at least $|A| (\delta(G) - n/3 -1) \geq (n/3) \cdot (\omega n - 1) \geq \omega n^2/4$ crossing edges. It follows that every vertex $x \in A$ lies in at most $\deg(x)/2$ crossing edges, as otherwise moving $a$ from $A$ to $B$ would yield a partition of $V(G)$ with parts of size at least $n/3$ and with fewer crossing edges. So we must have $\delta(G[A]) \geq \delta(G)/2 \geq n/6 + \omega n/2$, and the same argument with $B$ in place of $A$ shows that $\delta(G[B]) \geq n/6 + \omega n/2$.

Our proof now diverges according to whether we are proving conclusion~(a) or conclusion~(b) of Theorem~\ref{main}. For conclusion~(a) we simply choose arbitrarily a set $S$ of at most four vertices of $G$ so that $|A \sm S|$ and $|B \sm S|$ are each divisible by 3. For conclusion~(b) we instead use our additional assumptions that $G$ has no divisibility barrier and that 3 divides $n$. Indeed, the latter implies that we must have one of the following three cases:
\begin{enumerate}[label=(\alph*), noitemsep]
\item $|A| \equiv |B| \equiv 0 \pmod 3$. In this case we take $S = \emptyset$.
\item $|A| \equiv 1 \pmod 3$ and $|B| \equiv 2 \pmod 3$. Since $(A, B)$ is not a divisibility barrier, either $G$ contains an $B$-triangle or a pair of vertex-disjoint $A$-triangles, and we take $S$ to be the vertices covered by some such triangle or pair of triangles.
\item $|A| \equiv 2 \pmod 3$ and $|B| \equiv 1 \pmod 3$. Since $(B, A)$ is not a divisibility barrier, either $G$ contains an $A$-triangle or a pair of vertex-disjoint $B$-triangles, and we take $S$ to be the vertices covered by some such triangle or pair of triangles.
\end{enumerate}
Observe that in all cases we have $|S| \leq 6$ and that both $|A \sm S|$ and $|B \sm S|$ are divisible by $3$. The remaining part of the proof is the same for both cases.

Let $X_A \subseteq A$ consist of all vertices of $A$ with $\deg_{G[A]}(x) < n/3 + \omega n/2$. 
Then each vertex of $X_A$ is contained in more than $\omega n/2$ crossing edges, and since there are at most $\omega^2 n^2/40$ crossing edges in total, each with one vertex in $A$, it follows that
$|X_A| \leq  \omega n/20$. Since $\alpha(G) \leq \gamma n$ and $\delta(G[A]) \geq n/6 \geq 2|X_A| + |S| + \gamma n$ we may greedily form a triangle-tiling $\T_A$ of size at most $|X_A|$ in $G[A]$ which covers every vertex of $X_A$ but which does not intersect $S$. 
We then define $A' := A \sm (V(\T_A) \cup S)$, $G_A := G[A']$ and $n_A := |A'|$. Then $\delta(G_A) \geq  n/3 + \omega n/2 - |V(\T_A)| - |S| \geq n/3 + \omega n/3$, so $n/3 + \omega n/3 \leq n_A \leq 2n/3$. It follows that $G_A$ is a graph on $n_A$ vertices with $\delta(G_A) \geq n_A/2 + \omega n_A/2$ and $\alpha(G_A) \leq \gamma n \leq 3\gamma n_A$. Also $n_A$ is divisible by 3 (since 3 divides each of $|A \sm S|$ and $|V(\T_A)|$), so $G_A$ contains a perfect triangle-tiling $\T_A'$ by Theorem~\ref{BMSthm}.

By exactly the same argument with $B$ in place of $A$ we obtain a triangle-tiling $\T_B$ in $G[B]$ and a graph $G_B$ on vertex set $B' := B \sm (V(\T_B) \cup S)$ which contains a perfect triangle-tiling $\T_B'$. Finally, for conclusion~(a) observe that $\T := \T_A \cup \T_B \cup \T_A' \cup \T_B'$ is then a triangle-tiling in $G$ covering all vertices outside $S$, that is, all but at most four vertices of $G$, and for conclusion~(b) note that adding the triangle or triangles covering $S$ to $\T$ gives a perfect triangle-tiling in $G$.
\end{proof}

\section{Constructions and questions} \label{sec:examples}

Many of the ideas of this section are due to Balogh, Molla and Sharifzadeh~\cite{BMS}, but we include them here for completeness.
In the following constructions, we call an $n$-vertex 
triangle-free graph with sublinear independence number and minimum degree
$\omega(1)$ an \emph{Erd\H{o}s graph}, which we denote by $\ER(n)$.\footnote{
Strictly speaking, for any fixed $n$, both $\ER(n)$ and $\BE(n)$ (defined shortly) are families of graphs.}

To show that the minimum degree conditions of Theorem~\ref{main} and Corollary~\ref{krfree} are asymptotically tight, we
use the Erd\H{o}s graph to form a graph $G$ on $n$ vertices
with a space barrier, $\delta(G) \ge n/3 + \omega(1)$ and $\alpha(G) = o(n)$.
We form $G$ by taking the complete bipartite graph whose vertex classes $U$ and $V$ have sizes $2n/3 + 1$ and $n/3 - 1$ respectively,
and then placing copies of $\ER(|U|)$ and $\ER(|V|)$
on $U$ and $V$ respectively.
The graph $G$ formed in this way has $\delta(G) \geq n/3 + \omega(1)$ and sublinear independence number. Furthermore, $G$ is $K_5$-free since $G[U]$ and $G[V]$ are each triangle-free and, because $U$ is a space barrier, $G$ has no perfect triangle-tiling.

The previous example can be modified slightly to give lower bounds for the following question.

\begin{question}\label{q-kge4}
Let $k \geq 4$ and let $G$ be an $n$-vertex graph with $\alpha(G)=o(n)$. 
What is the best-possible minimum degree condition on $G$ that guarantees a perfect $K_k$-tiling in $G$?
\end{question}

The construction is slightly different depending on the parity of $k \ge 4$.
We start with the odd case, so let $k=2(\ell-1)+1$ for some integer $\ell \ge 3$. 
Consider the complete $\ell$-partite graph with one part $V_1$ of size $n/k-1$, another part $V_2$ of size $2n/k+1$ and 
the remaining parts $V_3, \dotsc, V_{\ell}$ each of size $2n/k$, 
and place the Erd\H{o}s graph $\ER(|V_i|)$ on each of the parts $V_i$.
When $k = 2\ell$ for some integer $\ell \ge 1$, 
the construction is essentially the same but we have one part of size $2n/k + 1$, one part of size $2n/k - 1$
and the remaining parts are each of size $2n/k$. In either case we obtain a graph $G$ with $\delta(G) \ge \left(1 - \frac{2}{k}\right)n + \omega(1)$, sublinear independence number and no $K_k$-factor.
It is worth noting that in the odd case the graph $G$ is $K_{k+2}$-free and in the even case $G$ contains
no $K_{k+1}$.

We feel that the following is another interesting related question.

\begin{question}\label{q-k4free}
Let $G$ be an $n$-vertex $K_4$-free graph with $\alpha(G) = o(n)$. 
What is the best-possible minimum degree condition on $G$ that guarantees a perfect triangle-tiling in~$G$? 
\end{question}

We use a modified version of the 
Bollob\'as-Erd\H{o}s graph~\cite{BE} to construct a $K_4$-free graph without a perfect triangle-tiling and with high minimum degree.
For every large even $n$, 
the Bollob\'as-Erd\H{o}s graph is an 
$n$-vertex, $K_4$-free graph with sublinear independence number, which we denote by $\BE(n)$.
The vertex set of $\BE(n)$ is the disjoint union of two sets $V_1$ and $V_2$ of the 
same order such that the 
graphs $G[V_1]$ and $G[V_2]$ are triangle-free and
every vertex in $V_1$ has at least $(1/4 - o(1))n$ neighbors in $V_2$
and every vertex in $V_2$ has at least $(1/4 - o(1))n)$ neighbors in $V_1$.
To construct our example, start with $\BE(4n/3 + 2)$
and then remove a randomly selected subset of size 
$n/3 + 2$ from one of the two parts.
Note that the two parts now have sizes
$n/3 - 1$ and $2n/3 + 1$, the resulting graph clearly is $K_4$-free and 
since the larger part is a space barrier, it has no perfect triangle-factor.
Furthermore, with high probability, the minimum degree is $(1/6 - o(1))n$.
We conjecture that $(1/6 + o(1))n$ is the proper minimum degree condition.

\begin{conjecture}\label{conj}
For every $\omega > 0$ there exist $\gamma, n_0 > 0$ such that every $K_4$-free graph on $n \geq n_0$ vertices with $\delta(G) \geq n/6 + \omega n$ and $\alpha(G) \leq \gamma n$ contains a perfect triangle-tiling.
\end{conjecture}

Using methods similar to those used in our proof of Theorem~\ref{main} we can show that every graph $G$ which satisfies the conditions of Conjecture~\ref{conj} has a triangle-tiling covering almost all of the vertices of $G$. More precisely, we can show that for $1/n \ll \gamma \ll \omega$, if $G$ is a 
$K_4$-free graph on $n$ vertices with $\delta(G) \ge (1/6 + \omega)n$  and $\alpha(G) \le \gamma n$, 
then $G$ contains a triangle-tiling which covers all but at most $\omega n$ vertices. What follows is a brief sketch of the argument.

Apply Theorem~\ref{reglem} with 
$\gamma \ll \eps \ll d \ll \omega$ to obtain a spanning subgraph $G' \subseteq G$, an exceptional set $V_0$ and clusters $V_1, \dots, V_k$ of equal size $m$.  Define the corresponding reduced graph $R$ on vertex set $[k]$ in the usual way. 
The fact that $G$ is $K_4$-free implies 
the following two important facts about these clusters and the graph $R$. 
(These facts were first observed by Szemer\'edi in~\cite{Sz}.) 
\begin{enumerate}[label=(\alph*), noitemsep]
  \item there is no pair $i, j \in [k]$ for which $G'[V_i, V_j]$ is $(1/2+d, \eps)$-regular, and 
  \item $R$ is triangle-free.
\end{enumerate}
Using a standard argument, it is not hard to see that~(a) and the fact that $\delta(G) \ge (1/6 + \omega)n$ together imply that $\delta(R) \ge k/3$.
So $R$ must be connected, as otherwise Mantel's theorem would give a triangle in the smallest connected component of $R$, contradicting~(b).
By a result of Enomoto, Kaneko and Tuza~\cite{EKT}, the fact that $R$ is a connected graph on $k$ vertices with $\delta(R) \ge k/3$
implies that $R$ contains $\floor{|R|/3}$ vertex-disjoint copies of $P_2$ (the path on three vertices).
In a manner similar to the proof of Lemma~\ref{tiling}, for each such path $ijk$ we can use the fact that $\alpha(G) \le \gamma n$ 
to greedily construct a triangle-tiling covering all but at most $3.1 \eps m$ of the vertices of $G[V_i \cup V_j \cup V_k]$, where each triangle has one vertex in $V_j$, the central cluster in the path,
and the other two vertices either both in $V_i$ or both in $V_k$.
The union of these $\floor{|R|/3}$ triangle-tilings is then a triangle-tiling in $G$ which covers all but at most $\omega n$ vertices.

\medskip
We can generalize Question~\ref{q-k4free} in the following way.
\begin{question}
Let $k \geq 3$ and let $G$ be an $n$-vertex $K_{k+1}$-free graph with $\alpha(G)=o(n)$. 
What is the best-possible minimum degree condition on $G$ that guarantees a perfect $K_k$-tiling in $G$? 
\end{question}

When $k$ is even, we have previously shown 
that the minimum degree must be at least $\left(\frac{k-2}{k} + o(1)\right)n$.
When $k = 2\ell + 1 \ge 5$, 
we form $G$ by starting with the complete $\ell$-partite graph
that has
one part $V_1$ of size $3n/k + 1$,
one part $V_2$ of size $2n/k - 1$, and 
the remaining parts, $V_3,\dotsc,V_{\ell}$, each of size $2n/k$.
In $V_1$, we place $\BE(|V_1|)$ on $V_1$, 
and, for every $2 \le i \le \ell$, 
we place a copy of $\ER(|V_i|)$ on $V_i$.
We then have $\delta(G) \geq \left(\frac{k-3}{k} + \frac{1}{4} \cdot \frac{3}{k} - o(1)\right)n =  
\left(\frac{4k-9}{4k} - o(1)\right)n$.
Furthermore, $G$ has sublinear independence number, is $K_{k+1}$-free, and 
has no perfect $K_{k}$-tiling, because each copy of $K_k$ in $G$ has at most $3$ vertices in $V_1$.

\medskip

Finally, for $r \ge 3$, $\omega, \gamma > 0$ and sufficiently large $n$, we 
give the construction of 
$G := G_{\text{RT}}(n, r, \omega, \gamma)$ from Theorem~\ref{rt-thm}(b).
For odd $r$ the construction was first given in~\cite{ES} and for even $r$ the construction is from~\cite{EHSS}.
We say that a partition $V_1, \dotsc, V_\ell$ of the vertices of a graph is \textit{equitable} if 
$||V_i| - |V_j|| \le 1$ for all $1 \le i < j \le \ell$.

When $r = 2\ell + 1$ is odd, we let $V_1, \dotsc, V_\ell$ be an equitable partition of $V(G)$ and form the complete $\ell$-partite graph with vertex classes $V_1, \dots, V_\ell$.
For every $i \in [\ell]$, we then place a copy of $\ER(|V_i|)$ on $V_i$, so
$$\delta(G) \ge n - \ceiling{\frac{n}{\ell}} \ge \left(\frac{r - 3}{r-1} - \omega\right)n.$$
We can assume that $n$ is large enough so that for each $i \in [\ell]$
the independence number of $G[V_i]$ is at most $\gamma n$, which implies that 
$\alpha(G) \le \gamma n$. Note that $G$ is $K_r$-free, as $G[V_i]$ is $K_3$-free for $i \in [\ell]$.

When $r = 2\ell$ is even, we let $U_1, \dotsc, U_{3\ell-2}$ be a equitable partition of $V(G)$,
so $|U_i| \in \left\{\floor{\frac{2n}{3r - 4}}, \ceiling{\frac{2n}{3r - 4}}\right\}$ for every $i \in [3\ell - 2]$.
Let 
$$V_1 := U_1 \cup U_2 \cup U_3 \cup U_4 \qquad \text{ and } \qquad V_i := U_{3i - 1} \cup U_{3i} \cup U_{3i + 1}
\quad \text{ for } 2 \le i \le \ell-1,$$
and form the complete $(\ell-1)$-partite graph with vertex classes $V_1, \dots, V_{\ell-1}$.
On $V_1$, we then place a copy of $\BE(|V_1|)$ and assume $n$ is large enough 
so that $G[V_1]$ has minimum degree at least 
\begin{equation*}
  \left(\frac{1}{4} - \omega\right)|V_1| \ge |V_1| - \left(\frac{6}{3r - 4} + \omega\right)n
\end{equation*}
and independence number at most $\gamma n$.
For every $2 \le i \le \ell-1$, we place a copy of $\ER(|V_i|)$ on $V_i$ and we ensure that $n$ is large enough so that
the independence number of $G[V_i]$ is at most $\gamma n$.
Because every vertex in $G$ is adjacent to all but at most 
$\left(\frac{6}{3r - 4} + \omega\right)n$ vertices of $G$, we have that 
\begin{equation*}
  \delta(G) \ge \left(\frac{3r - 10}{3r - 4} - \omega\right)n.
\end{equation*}
Furthermore, $\alpha(G) \le \gamma n$ and $G$ is $K_r$-free as $G[V_1]$ is $K_4$-free and each of 
the subgraphs $G[V_2], \dotsc, G[V_{\ell-1}]$ is $K_3$-free.

\section{Appendix} \label{appendix}

The purpose of this appendix is to prove Lemma~\ref{random-slicing}. 
The lemma is essentially a corollary to the following two theorems
of Kohayakawa and R\"odl~\cite{KR}. For this we use the following notation: let $G$ be a bipartite
graph with vertex classes $A$ and $B$, and define $d := d(G[A, B])$. Then for any $\eps$ we define $D_{AB}(\eps)$ to be the graph with vertex set $A$ in which $x,x' \in A$ are adjacent if and only if 
  \begin{equation*}
    |N_G(x)|, |N_G(x')| > (d - \eps)|B| \qquad \text{and} \qquad |N_G(x) \cap N_G(x')| < (d + \eps)^2|B|. 
  \end{equation*}

\begin{theorem}[{\cite[Theorem 45]{KR}}]\label{kr-sufficient}
  Let $0 < \eps < 1$,
  and let $G[A,B]$ be a bipartite graph with $|A| \ge 2/\eps$. 
  If $e(D_{AB}(\eps)) > (1 - 5\eps)|A|^2/2$, then $G[A,B]$ is $(d, (16 \eps)^{1/5})$-regular, where $d := d(G[A, B])$.
\end{theorem}

\begin{theorem}[{\cite[Theorem 46]{KR}}]\label{kr-necessary}
  Let $0 < \eps < 1$,
  and let $G[A, B]$ be a bipartite graph with $|B| \ge 1/d$, where $d := d(G[A,B])$. 
  If $G[A,B]$ is $(d, \eps)$-regular, then $e(D_{AB}(\eps)) \ge (1 - 8\eps)|A|^2/2$.
\end{theorem}

The following two similar lemmas do most of the remaining work required to complete the proof. 
\begin{lemma}\label{lem:chernoff-slice-bigraph}
  Suppose that $1/n \ll \xi \ll \xi'$ and that $1/n \ll \beta$.
  Let $G[A,B]$ be a bipartite graph such that $|A|, |B| \le n$, and let
  $x_1,\dotsc,x_s$ and $y_1,\dotsc,y_t$ be positive integers each of size at least $\beta n$ such that
  $\sum_{i \in [s]} {x_i} \le |A|$ and $\sum_{j \in [t]} {y_j} \le |B|$.
  If $\{X_1, \dotsc, X_s\}$ is a collection of disjoint subsets of $A$ and 
  $\{Y_1, \dotsc, Y_t\}$ is a collection of disjoint subsets of $B$ 
  with $|X_i| = x_i$ and $|Y_j| = y_j$ for all $i \in [s]$ and $j \in [t]$ 
  selected uniformly at random from all such collections, 
  then, with probability at least $1 - e^{-\Omega(n)}$, 
  for every $i \in [s]$, $j \in [t]$, $x,x' \in A$ and $y,y' \in B$ we have
  \begin{enumerate}[label=(\alph*), noitemsep]
    \item
      $|N_G(x) \cap Y_j|/y_j = |N_G(x)|/|B| \pm \xi$,  
    \item
      $|N_G(y) \cap X_i|/x_i = |N_G(y)|/|A| \pm \xi$,
     \item
      $|N_G(x) \cap N_G(x') \cap Y_j|/y_j = |N_G(x) \cap N_G(x')|/|B| \pm \xi$,
    \item
      $|N_G(y) \cap N_G(y') \cap X_i|/x_i = |N_G(y) \cap N_G(y')|/|A| \pm \xi$, and
    \item
      $d(G[X_i, Y_j]) = d(G[A,B]) \pm \xi'.$
  \end{enumerate}
\end{lemma}
\begin{proof}
  Note that the at most 
  $t(|A| + |A|^2) + s(|B| + |B|^2) \le 2 \beta^{-1}(n + n^2)$ random variables of the form
  $|N_G(x) \cap Y_j|$,
  $|N_G(y) \cap X_i|$,
  $|N_G(x) \cap N_G(x') \cap Y_j|$, and
  $|N_G(y) \cap N_G(y') \cap X_i|$,  
  where $i \in [s]$, $j \in [t]$, 
  $x,x' \in A$ and $y,y' \in B$,
  are hypergeometrically distributed,
  so the fact that (a)-(d) hold with probability
  $1 - e^{\Omega(n)}$ follows directly from Theorem~\ref{chernoff} by taking a union bound.
  For (e), let $\ell := \xi^{-1}/2$ and define 
$D_k := \{v \in A : 2(k - 1) \xi \le |N(v)|/|B| < 2k \xi\}$
for each $k \in [\ell]$. Then, with probability $1 - e^{\Omega(n)}$, 
for every $i \in [s]$ and $k \in [\ell]$, 
  we have that 
  \begin{equation*}
    \frac{|D_k \cap X_i|}{x_i} = \frac{|D_k|}{|A|} \pm \xi^2.
  \end{equation*}
  Fix a choice of $X_1, \dotsc, X_s$ and $Y_1, \dotsc, Y_t$, for which (a)-(d) hold and this event occurs.
  Note that for every $k \in [\ell]$, $v \in D_k$, and $j \in [t]$,
  \begin{equation*}
    \frac{|N_G(v)|}{|B|} = (2k - 1)\xi \pm \xi \qquad \text{ so } \qquad
    \frac{|N_G(v) \cap Y_j|}{y_j} = (2k - 1)\xi \pm 2\xi. 
  \end{equation*}
  We compute $d(G[A, B])$ to be 
  \begin{equation*}
    \frac{1}{|A|} \sum_{k \in [\ell]}\sum_{v \in D_k} \frac{|N_G(v)|}{|B|} = 
    \sum_{k \in [\ell]} \left(((2k - 1)\xi \pm \xi) \cdot \frac{|D_k|}{|A|} \right)= 
    \left(\sum_{k \in [\ell]}(2k - 1)\xi\frac{|D_k|}{|A|}\right) \pm \xi.
  \end{equation*}
  Then for any $i \in [s]$ and $j \in [t]$ we have 
  \begin{equation*}
    \begin{split}
    d(G[X_i, Y_j]) &= 
	\frac{1}{x_i}\sum_{k \in [\ell]} \sum_{v \in D_k \cap X_i} \frac{|N_G(v) \cap Y_j|}{y_j} = 
    \sum_{k \in [\ell]} \left(((2k - 1)\xi \pm 2\xi) \cdot \left(\frac{|D_k|}{|A|} \pm \xi^2\right)\right) \\
    &= \left(\sum_{k \in [\ell]}(2k - 1)\xi\frac{|D_k|}{|A|}\right) \pm (\ell^2\xi^3 + 2\xi + 2\ell \xi^3)     =  d(G[A, B]) \pm \xi',
  \end{split}
  \end{equation*}
  so (e) holds.
\end{proof}

\begin{lemma}\label{lem:chernoff-slice-graph}
  Suppose that $1/n \ll \xi \ll \xi'$ and $1/n \ll \beta$, and that
  $x_1,\dotsc,x_s$ are positive integers each of size at least $\beta n$ such that
  $\sum_{i \in [s]} {x_i} \le n$.
  If $G$ is a graph on $n$ vertices and 
  $\{X_1, \dotsc, X_s\}$ is a collection of disjoint subsets of $V(G)$ 
  with $|X_i| = x_i$ for all $i \in [s]$ 
  selected uniformly at random from all such collections, 
  then, with probability at least $1 - e^{-\Omega(n)}$, 
  for every $i \in [s]$ and $x,x' \in V(G)$ we have
  \begin{enumerate}[label=(\alph*), noitemsep]
    \item
      $|N_G(x) \cap X_i|/x_i = |N_G(x)|/n \pm \xi$,  
     \item
      $|N_G(x) \cap N_G(x') \cap X_i|/x_i = |N_G(x) \cap N_G(x')|/n \pm \xi$, and
    \item
      $2e(G[X_i])/x_i^2 = 2e(G)/n^2 \pm \xi'.$
  \end{enumerate}
\end{lemma}
\begin{proof}
  It is straightforward to modify the proof of Lemma~\ref{lem:chernoff-slice-bigraph} to prove
  this lemma; we omit the details.
\end{proof}

Now we give the proof of Lemma~\ref{random-slicing}.
\begin{proof}[Proof of Lemma~\ref{random-slicing}]
Introduce a new constant $\eta$ with $1/n \ll \eta \ll \eps$.
  Suppose that $G[A, B]$ is $(\geq \!\! d, \eps)$-regular, let $d^* := d(G[A, B])$, so $d^* = d \pm \eps$, and define 
  $D := D_{AB}(\eps)$.
  Note that, by Theorem~\ref{kr-necessary}, we have that $2e(D)/|A|^2 \ge 1 - 8\eps$. We apply 
  Lemma~\ref{lem:chernoff-slice-bigraph} to $G[A, B]$ and Lemma~\ref{lem:chernoff-slice-graph} to $D$, 
  with $\xi'$ replaced by $\eta$ in each case, to find that with probability $1 - e^{-\Omega(n)}$ our random selection satisfies the conclusions of each of these lemmas. We fix such an outcome of our random selection, and consider any $i \in [s]$ and $j \in [t]$. Define $d_{ij} := d(X_i, Y_j)$, so
  $d_{ij} = d^* \pm \eta$, and
  \begin{equation}\label{eq:dij_size}
    d_{ij} = d \pm (\eps + \eta).
  \end{equation}
  We also have that
  $$\frac{2e(D[X_i])}{x_i^2} \ge \frac{2e(D)}{|A|^2}- \eta \geq  1 - 8\eps - \eta \ge 1 - 5(2 \eps).$$ 
  Recall that, if $xx' \in E(D[X_i])$, then 
  $$\frac{|N_G(x)|}{|B|}, \frac{|N_G(x')|}{|B|} > d^* - \eps \text{ and } \frac{|N_G(x) \cap N_G(x')|}{|B|} < (d^* + \eps)^2,$$
  so
  $$\frac{|N_G(x) \cap Y_j|}{y_j},\frac{|N_G(x') \cap Y_j|}{y_j} > (d^* - \eps) - \eta > d_{ij} - 2\eps,$$
  and, as we can assume $\eta$ is small enough so that $\eta^{1/2} + \eta < \eps$,
  $$\frac{|N_G(x) \cap N_G(x') \cap Y_j|}{y_j} < (d^* + \eps)^2 + \eta < (d_{ij} + \eta + \eps)^2 + (\eps - \eta)^2 < (d_{ij} + 2\eps)^2.$$
  This proves that $xx' \in E(D_{X_iY_j}(2\eps))$, so $D$ is a subgraph of $D_{X_iY_j}(2\eps)$.
  Therefore, by Lemma~\ref{kr-sufficient} with $d$ and $\eps$ replaced by $d_{ij}$ and $2 \eps$, respectively,
  $G[X_i, Y_j]$ is $(d_{ij}, (32 \eps)^{1/5})$-regular, and
  is therefore $(d, (32 \eps)^{1/5} + 2\eps)$-regular,
  because, by \eqref{eq:dij_size}, $d = d_{ij} \pm 2\eps$.
  Since we can assume that $\eps$ 
  is small enough so that $(32 \eps)^{1/5} + 2\eps \le (33 \eps)^{1/5}$, it follows that 
  $G[X_i, Y_j]$ is $(d, (33 \eps)^{1/5})$-regular.

  Clearly, if $G[A, B]$ is $(d, \eps)$-super-regular,
  then, by (a) and (b) of Lemma~\ref{lem:chernoff-slice-bigraph}, 
  we can also ensure that $G[X_i, Y_j]$ is $(d, (33 \eps)^{1/5})$-super-regular for each $i \in [s]$ and $j \in [t]$. 
\end{proof}

\end{document}